\documentclass[a4paper,11pt]{amsart}

\usepackage{import}
\usepackage[percent]{overpic}

\usepackage{xcolor, hyperref}
\usepackage[utf8x]{inputenc}%encodage utf8
\usepackage[english]{babel}
\usepackage[T1]{fontenc}%permet les accents et adopte les regles typographiques francaises, chapitre théorème...package à rajouter si la lanque est le francais
\usepackage{graphicx}%permet d'inserer des images
\usepackage{enumerate}%contient des options de l'environnement enumerate
\usepackage{amsthm, amsmath, amssymb}%ajoute des options/symboles... de maths de l'ams
\usepackage{fancyhdr}%package pour personnaliser les pieds de page et entetes
\usepackage{lmodern}

\usepackage{xypic}

\usepackage{pgfplots}
\pgfplotsset{compat=1.15}
\usepackage{mathrsfs}
\usetikzlibrary{arrows}

\usepackage{pgfgantt}

\usepackage{geometry}%pour régler les marges
\usepackage{mathtools, stmaryrd}
\usepackage{xparse}

\newtheoremstyle{Largetheorem}  % nouveau style de theoreme : theoreme ecrit en gros
{3pt}       % Space above 
{3pt}       % Space below
{}      % Body font
{}          % Indent amount 
{\Large \bfseries}  % Theorem head font
{.}         % Punctuation after theorem head
{.5em}      % Space after theorem head
{}          % Theorem head spec

\theoremstyle{Largetheorem}%theoreme ecrit en gros, texte en italique

\theoremstyle{definition}%texte pas en italique

\newtheorem*{remark}{Remark}

\theoremstyle{plain} %texte en italique
\newtheorem{theo}{Theorem}[section]
\newtheorem{lemma}[theo]{Lemma}
\newtheorem{prop}[theo]{Proposition}
\newtheorem{coro}[theo]{Corollary}

\newtheorem{conjecture}[theo]{Conjecture}

\newcommand{\R}{\mathbb{R}}
\newcommand{\C}{\mathbb{C}}
\newcommand{\Z}{\mathbb{Z}}
\newcommand{\N}{\mathbb{N}}
\newcommand{\Q}{\mathbb{Q}}

\renewcommand{\leq}{\leqslant}
\renewcommand{\geq}{\geqslant}
\renewcommand{\d}[1]{{\ensuremath{\,\text{d}#1}}}

\newcommand{\gauss}[2]{\begin{bmatrix}#1\\#2\end{bmatrix}_\lambda}

\def\d{{\rm d}}
\def\i{{\rm i}}
\def\ds{\displaystyle}
\def\e{{\rm e}}

\def\O{{\mathcal O}}

\def\eps{{\varepsilon}}

\newcommand{\Qp}{\mathbb{Q}_p}

\usepackage{eurosym}
\usepackage{array}

\usepackage{multirow}

\def\ob{{\llbracket}}
\def\cb{{\rrbracket}}

\def\Sing{{\rm Sing}}

\def\F{{\mathbb F}}

\def\resultant{\operatorname{Res}}

\def\Qbar{\overline\Q}

\begin{author}[X.~Buff]{Xavier Buff}
\address{ %
Univ Toulouse, INSA Toulouse, CNRS, IMT, Toulouse, France}
\email{xavier.buff@math.univ-toulouse.fr}
\end{author}

\begin{author}[V.~Huguin]{Valentin Huguin}
\address{LAMFA (Laboratoire Amiénois de Mathématique Fondamentale et Appliquée), UMR CNRS 7352, Université de Picardie Jules Verne, 33 rue Saint-Leu, 80039 Amiens, France}
\email{valentin.huguin@u-picardie.fr}
\end{author}

\begin{author}[L.~Vivas]{Liz Vivas}
\address{Department of Mathematics, The Ohio State University, Columbus, OH}
\email{vivas@math.osu.edu}
\end{author}

\thanks{The first and second authors were partially supported by the French National Research Agency under the project DynAtrois (ANR-24-CE40-1163). 
They would also like to thank the Isaac Newton Institute for Mathematical Sciences, Cambridge, for support and hospitality during the programme Complex dynamics: interactions and influences. This work was supported by EPSRC grant EP/Z000580/1 and NSF grant DMS 2453797.}

\begin{document}

\title{An Arithmetic Approach to Parabolic Multiplicity in Complex Dynamics}

\begin{abstract}
When $\omega$ is a primitive $n$-th root of unity,  the quadratic polynomial $F(z) = \omega z(1-z)$ and the entire map $F(z) = \omega z {\rm e}^{-z}$ both have a parabolic fixed point at $0$. Their parabolic multiplicity is equal to $1$, that is, $F^{\circ n}(z) = z \bigl(1+c z^n+{\mathcal O}(z^{n+1})\bigr)$ with $c \neq 0$. The classical proof of this fact is transcendental. We present an arithmetic proof which may be extracted from \cite{bst} in the transcendental case and requires working in ${\mathbb Z}/(n-1){\mathbb Z}$, and which is new in the polynomial case and requires working in the $p$-adic field ${\mathbb Q}_p$ for a suitable prime $p$ such that the order of $2$ in $({\mathbb Z}/p{\mathbb Z})^\times$ is exactly $n$. 
\end{abstract}

\maketitle

\section{Introduction}

Assume $n\geq 1$ is an integer and let $\omega\in \C$ be a primitive $n$-th root of unity.
Let $F:\C\to \C$ be an entire map fixing $0$ with derivative $\omega$. If $F$ is not a rotation, then there exists an integer $\nu\geq 1$, called the parabolic multiplicity of $F$ at $0$, such that 
\[F^{\circ n}(z) = z  \bigl(1+ cz^{\nu n } + \O(z^{\nu n+1})\bigr)\quad \text { with } \quad c \neq 0.\]
The parabolic multiplicity $\nu$ is closely related to the number of singular values of $F$. 
A point $z\in \C$ is a regular value of $F$ if there exists an open neighborhood $V$ of $z$ in $\C$ such that the restriction $F: F^{-1}(V)\to V$ is a covering map. The point $z$ is a singular value of $F$ otherwise. Denote by $\Sing(F)$ the set of singular values of $F$. For example: 
\begin{itemize}
\item if $F(z) = \omega z(1-z)$, then $\Sing(F) = \{\frac{\omega}{4}\}$;
\item if $F(z) = \omega z\e^{-z}$, then $\Sing(F) = \{0;\frac{\omega}{\e}\}$.
\end{itemize}
The following theorem goes back to Fatou in the polynomial case \cite{f1,f2} (see also \cite[\S 10]{m} for a modern exposition) and extends readily to entire maps (see \cite{be} for example). 
\begin{theo}\label{theo:fatou}
If $\omega$ is a root of unity and if $F$ is an entire map, not a rotation, fixing $0$ with multiplier $\omega$ and parabolic multiplicity $\nu$, then $0$ attracts the orbit of at least $\nu$ singular values of $F$. 
\end{theo}

\begin{coro}\label{cor:quad}
If $\omega$ is a root of unity and $F(z) = \omega z(1-z)$, then the parabolic multiplicity of $F$ at $0$ is equal to $1$. 
\end{coro}

\begin{coro}\label{cor:exp}
If $\omega$ is a root of unity and $F(z) = \omega z\e^{-z} $, then the parabolic multiplicity of $F$ at $0$ is equal to $1$. 
\end{coro}

The classical proof of Theorem~\ref{theo:fatou} is transcendental. In this paper, we present arithmetic proofs of these two corollaries.  An arithmetic proof of Corollary~\ref{cor:exp} may be extracted from \cite[Proposition B.1]{bst}. Somewhat surprisingly, obtaining an arithmetic proof is easier for the transcendental map $\omega ze^{-z}$ than for the quadratic polynomial $\omega z(1-z)$.

In the transcendental case, our proof goes as follows. We consider the family of transcendental entire maps $F_\lambda(z) = \lambda z \e^{-z}$ with $\lambda \in \C^\ast$. Given an integer $n\geq 1$, we may write 
\[F_\lambda^{\circ n}(z)  = \lambda^n z \sum_{k\geq  0} c_{n,k}(\lambda) z^k\quad \text{with}\quad c_{n,k}(\lambda)\in \Q[\lambda].\]
Let $\Omega_n\subset \Qbar$ be the set of primitive $n$-th roots of unity and let $\Phi_n(\lambda)\in \Z[\lambda]$ be the $n$-th cyclotomic polynomial:
\[
\Phi_n(\lambda):=\prod_{\omega \in \Omega_n}(\lambda-\omega).
\]
We will prove that the resultant 
\[\beta_n := \resultant(\Phi_n,-(n-1)!c_{n,n}) = \prod_{\omega\in \Omega_n} \bigl(-(n-1)! c_{n,n}(\omega)\bigr),\]
is a nonzero integer, which implies that the parabolic multiplicity of $F_\omega$ at $0$ is equal to $1$ for any $\omega\in \Omega_n$. Our proof requires working in $\Z/(n-1)\Z$. Our result is the following. 

\begin{prop}\label{prop:beta}
Assume $F_\lambda(z) := \lambda  z \e^{-z}$,  $n\in \N^\ast$, 
\[F_\lambda^{\circ n}(z)  = \lambda ^n z \sum_{k\geq  0} c_{n,k}(\lambda) z^k
\quad\text{and}\quad \beta_n := \resultant(\Phi_n,-(n-1)!c_{n,n}).\]
Then, $\beta_n\in \Z^\ast$. More precisely, $\beta_1 = \beta_2 =1$, 
\begin{itemize}
\item if $n\geq 3$ is a power of a prime $p$, then $p^{n-1} \beta_n\equiv  1 \pmod{n-1}$ and
\item if $n\geq 5$ is not a prime power, then  $\beta_n\equiv  1 \pmod{n-1}$.
\end{itemize}
\end{prop}

In the case of the quadratic polynomial $\omega z(1-z)$, Yoccoz \cite{yoccoz} claims to have an arithmetic proof that when $\omega$ is a root of unity, the parabolic multiplicity at $0$ is equal to $1$. However, he makes an assertion that we do not know how to prove (see the discussion in \S\ref{sec:valueroots}). In this paper, we present a proof which is valid for  unicritical polynomials of arbitrary degree. 

The proof proceeds as follows. Let $d\geq 2$ be an integer and consider the family of unicritical polynomials 
\[F_\lambda(z) := \frac\lambda{d^2} \bigl(1- (1-dz)^d\bigr).\]
The polynomial $F_\lambda$ fixes $0$ with derivative $\lambda$. 
It has degree $d$, integer coefficients and a unique critical point at $z = \frac{1}{d}$. 
For $d = 2$, we have that $F_\lambda(z) = \lambda z (1-z)$.

As previously, given $n\in \N^\ast$, we may write 
\[F_\lambda^{\circ n}(z)  = \lambda^n z \sum_{k\geq  0} c_{n,k}(\lambda) z^k \quad \text{with}\quad c_{n,k}(\lambda)\in \Z[\lambda].\]
We need to show that the polynomial $c_{n,n}$ does not vanish on $\Omega_n$. The cases 
\begin{itemize}
\item $n=1$, 
\item $n=2$, and 
\item $n=6$ and $d=2$
\end{itemize}
can be handled by direct computation. By the Bang-Zsigmondy Theorem \cite{ba,z}, in all other cases, there exists a prime number $p$ such that the order of $d$ in $(\Z/p\Z)^\times$ is exactly $n$. We shall work in the $p$-adic completion $\bigl(\Qp,|\cdot|_p\bigr)$ of $\Q$ for such a prime number $p$. Since $\Phi_n\in \Z[\lambda]$ is monic and irreducible over $\Q$ and since $c_{n,n}\in \Z[\lambda]$, if $c_{n,n}$ and $\Phi_n$ have a common zero in $\C$, then $\Phi_n$ divides $c_{n,n}$ in $\Z[\lambda]$, whence every root of $\Phi_n$ in $\Qp$ is also a root of $c_{n,n}$. The fact that $c_{n,n}$ and $\Phi_n$ have no common root in $\C$ is therefore a consequence of the following proposition. 

\begin{prop}\label{prop:padic}
Assume that $d\geq 2$ and  $n\geq 3$ are integers. Assume further that $p\geq 2$ is a prime number such that the order of $d$ in $(\Z/p\Z)^\times$ is exactly $n$. Then, the cyclotomic polynomial $\Phi_n$ has a unique root $\omega\in \Qp$ such that $|\omega-d|_p<1$. If 
\[F(z) =\frac{\omega}{d^2} \bigl(1- (1-dz)^d\bigr)\]
then 
\[F^{\circ n}(z) = z \bigl(1 + cz^n + \O(z^{n+1}) \bigr) \quad \text{with}\quad |c|_p\geq |d^n-1|_p\neq 0.\]
\end{prop}

Finally, in the case of quadratic polynomials, we will complement this result with observations, conjectures, and a few other results. Let $\varphi$ is the Euler totient function and $v_p$ be the $p$-adic valuation on $\Z$. 

\begin{prop}
Assume $F_\lambda(z) = \lambda z(1-z)$, $n\in\N^\ast$,
\[F_\lambda^{\circ n}(z)  = \lambda^n z \sum_{k\geq  0} c_{n,k}(\lambda) z^k\quad \text{and}\quad 
\alpha_n:= \resultant(\Phi_n,-c_{n,n}).\]
\begin{itemize}
\item If $n = 2^s-1$ for some integer $s\geq 1$, then $\alpha_n$ is an odd integer. 
\item If $n=2(2^s-1)$ for some integer $s\geq 1$, then $v_2(\alpha_n) = s\varphi(n)$.
\item If $n = 2^r+2^s-1$ for some integers $r>s\geq 1$, then $v_2(\alpha_n)=\varphi(n)$. 
\item For every integer $n\geq 1$,  $n$ divides $\alpha_n$.
\item If $n$ is a prime number, then $\alpha_n \equiv n \pmod{n^2}$. 
\end{itemize}
\end{prop}

\section{Preliminaries}

Before embarking into the study of the families $\lambda z (1-z)$ and $\lambda z \e^{-z}$, let us first recall some known facts. 

\subsection{Rising factorials\label{sec:factorials}}

Given a real number $x\in \R$ and an integer $n\geq 0$, we consider the rising factorial
\[x^{(n)} :=\prod_{j=0}^{n-1} (x+j) = x(x+1)\cdots (x+n-1).\]
By convention, $x^{(0)} = 1$ since an empty product is equal to $1$. 
The following Vandermonde-type identity for rising factorials is well known (see for example \cite{gkp}). 

\begin{lemma}\label{lem:factorials}
For every integer $m\in \N$ and every $x,y\in \R$, 
\[\sum_{j=0}^{m} \binom{m}{j} x^{(j)}y^{(m-j)} =(x+y)^{(m)}. \]
\end{lemma}

\begin{proof}
Fix $x\in \R$ and $y\in \R$ and for $m\in \N$, set
\[S_m := \sum_{j=0}^{m}\binom{m}{j}  x^{(j)}y^{(m-j)}. \]
We prove by induction on $m\in \N$ that
\[S_m=(x+y)^{(m)}.\]
For $m=0$, both sides are equal to $1$.
Assume that the identity holds for some $m\geq0$. Using the Pascal  Identity
\[\binom{m+1}{j}=\binom{m}{j}+\binom{m}{j-1},\]
we obtain 
\[S_{m+1} = \sum_{j=0}^{m+1} \binom{m}{j} x^{(j)} y^{(m+1-j)} + \sum_{j=0}^{m+1} \binom{m}{j-1} x^{(j)} y^{(m+1-j)}.\]
In the first sum, the term $j=m+1$ vanishes since $\ds \binom{m}{m+1}=0$. 
In the second sum, the term $j=0$ vanishes since $\ds \binom{m}{-1}=0$.
Hence, setting $k = j-1$, 
\[S_{m+1} =  \sum_{j=0}^{m} \binom{m}{j} x^{(j)} y^{(m-j)} (y+m-j) + \sum_{k=0}^{m}
\binom{m}{k} x^{(k)} y^{(m-k)}(x+k) .\]
Renaming $k$ as $j$ and combining the two sums gives
\[S_{m+1} =  \sum_{j=0}^{m} \binom{m}{j} x^{(j)} y^{(m-j)} (x+y+m) = (x+y+m) S_m = (x+y)^{(m+1)}\]
by the induction hypothesis. 
\end{proof}

\subsection{Cyclotomic polynomials}

The discriminant of the polynomial $\lambda^n-1\in \Z[\lambda]$ is $\pm n^n$. This polynomial therefore has no square factor. Since for all $d$ dividing $n$, $\lambda^d-1$ divides $\lambda^n-1$ in $\Z[\lambda]$, and since those polynomials are monic, there exists a sequence $(\Phi_n)_{n\in \N^\ast}$ of monic polynomials in $\Z[\lambda]$ such that for all $n\in \N^\ast$, 
\[\lambda^n-1 = \prod_{d|n} \Phi_d(\lambda).\]
The polynomial $\Phi_n$ is the $n$-th cyclotomic polynomial. Its degree is $\varphi(n)$, where $\varphi:\N^\ast\to\N^\ast$ is the Euler totient function. 

Let $\Omega_n\subset \Qbar$ be the set of primitive $n$-th roots of unity. Then,
\[\Phi_n(\lambda) = \prod_{\omega\in \Omega_n} (\lambda - \omega).\] 

Let $\mu:\N^\ast\to  \{-1, 0, 1\}$ be the Möbius function defined by:
\begin{itemize}
    \item $\mu(1) = 1$,
    \item if $n$ is a square-free positive integer with $j$ distinct prime factors, then $\mu(n) = (-1)^j$,
    \item if $n$ has a squared prime factor, then $\mu(n) = 0$.
\end{itemize}
Then, following equality holds in $\Q(\lambda)$: 
\[\Phi_n(\lambda) = \prod_{d|n} (\lambda^d-1)^{\mu(n/d)}.\]

The irreducibility of $\Phi_n$ over $\Q$ was established by Gauss for $n$ prime, and extended to all $n\geq 1$ by Kronecker (see~\cite{lang} for a modern treatment).

\begin{theo}
For every $n\geq 1$, the cyclotomic polynomial $\Phi_n$ is irreducible over $\Q$. 
\end{theo}

The following theorem, established by Apostol~\cite{apostol}, is now a classical tool in the study of cyclotomic fields and has since been used and rediscovered in a variety of contexts.

\begin{theo}\label{theo:apostol}
 For all $m,n\in\N^\ast$ with $m<n$, we have
\[\resultant(\Phi_m,\Phi_n) = 
\begin{cases} 
p^{\varphi(m)}&\text{if }n = mp^r\text { with } p\geq 2 \text{ prime and }r\in\N^\ast, \\ 
1& \text{otherwise.}
\end{cases}\]
\end{theo}

We will use consequences of the well-known fact below concerning cyclotomic polynomials (see \cite{washington}).

\begin{lemma}\label{lem:phimq}
For every prime number $p\geq 2$ and every $m\in\N^\ast$,
\[
\Phi_{mp}(\lambda) =
\begin{cases}
\Phi_m(\lambda^p) & \text{if } p\mid m,\\
\dfrac{\Phi_m(\lambda^p)}{\Phi_m(\lambda)} & \text{otherwise.}
\end{cases}
\]
\end{lemma}

\begin{coro}\label{cor:phikpr}
For every prime number $p\geq 2$, every $k\in\N^\ast$ not divisible by $p$ and every $r\in\N^\ast$,
\[\Phi_{kp^r}(\lambda) = \frac{\Phi_k(\lambda^{p^r})}{\Phi_k(\lambda^{p^{r-1}})}.\]
\end{coro}

\begin{coro}\label{cor:valuephin1}
For every $n\geq 2$, 
\[\Phi_n(1)=\begin{cases} p &\text{if }n=p^r,\text{ with }p\geq 2\text{ prime and }r \in\N^\ast, \\
1 &\text{otherwise.}
\end{cases}\]
\end{coro}

\subsection{Gaussian polynomials}

For $n\in \N$, let $\pi_n\in \Z[\lambda]$ denote the polynomial 
\[\pi_n := \prod_{j=1}^n (1-\lambda^j),\]
where the empty product is understood to be equal to 1, so that $\pi_0 = 1$. 

Given integers $0 \leq k \leq n$, it is well known that the binomial coefficient 
\[\binom{n}{k} :=  \frac{n!}{k!(n-k)!}\]
is an integer. This may be deduced from the Pascal relation
\[\binom{n}{k} = \binom{n-1}{k-1} + \binom{n-1}{k}.\]
Similarly, the Gaussian polynomial
\[\gauss{n}{k}:=\frac{\pi_n}{\pi_k\pi_{n-k}}
\]
is a polynomial with integer coefficients. This may be deduced from the Gauss relation
\[\gauss{n}{k}=\gauss{n-1}{k-1}+\lambda^k\gauss{n-1}{k}.\]
%
%
%\begin{lemma}
%For all integers $0\leq j\leq k$, we have that 
%\[\gauss{k}{j}(1) = \binom{k}{j}\quad \text{and}\quad \gauss{k}{j}'(1) = \binom{k}{j} \cdot  \frac{j(k-j)}{2}.\]
%\end{lemma}
%
%\begin{proof}
%For $k\geq 1$, let $\rho_k\in \Z[\lambda]$ and $\sigma_k\in \Z[\lambda]$ be defined by 
%\[\rho_k:= \frac{1-\lambda^k}{1-\lambda} = 1+\lambda + \cdots + \lambda^{k-1}\quad \text{and}\quad  \sigma_k:= \frac{\pi_k}{\pi_1^k} = \prod_{j=1}^k \rho_j.\]
%Note that
%\[\rho_k(1) = k, \quad \rho_k'(1) = 1+\cdots+(k-1) = \frac{(k-1)k}{2},\]
%\[\sigma_k(1) = k! \quad \text{and}\quad \frac{\sigma_k'(1)}{\sigma_k(1)} = \sum_{j=1}^k \frac{\rho_j'(1)}{\rho_j(1)} = \sum_{j=1}^k \frac{j-1}{2} = \frac{(k-1)k}{4}.\]
%Now, assume $0\leq j\leq k$.  If $j=0$ or $j=k$, then $\gauss{k}{j} = 1$ and the result is immediate. So, let us assume that $j\in \ob1,k-1\cb$. Then, 
%\[\gauss{k}{j} = \frac{\sigma_k}{\sigma_j \sigma_{k-j}}, \quad \text{thus}\quad \gauss{k}{j} (1)= \frac{k!}{j! (k-j)!} = \binom{k}{j}.\]
%In addition, 
%\begin{eqnarray*}
%\gauss{k}{j}'(1) &=& \gauss{k}{j}(1) \cdot \left(\frac{\sigma_k'(1)}{\sigma_k(1)} - \frac{\sigma_j'(1)}{\sigma_j(1)}-\frac{\sigma_{k-j}'(1)}{\sigma_{k-j}(1)}\right)\\
%&=& \binom{k}{j} \cdot \left(\frac{(k-1)k}{4} - \frac{(j-1)j}{4}-\frac{(k-j-1)(k-j)}{4}\right)\\
%& = & \binom{k}{j} \cdot  \frac{j(k-j)}{2}.\qedhere
%\end{eqnarray*}
%\end{proof}
Note that the roots of $\gauss{n}{k}$ are roots of unity; they belong to $\ds\bigcup_{j\in \ob2,n\cb} \Omega_j$. 

\begin{lemma}\label{lem:gausseuclid}
Assume $0\leq k\leq n$ and $q\geq 1$ are integers and $\omega \in \Omega_q$. Consider the Euclidean divisions $k=j q+r$ and $n=m q+s$ with $r, s\in \ob0,q-1\cb$. 
\begin{itemize}
\item If $r\leq s$, then $\omega$ is not a root of $\gauss{n}{k}$ and 
\[\gauss{n}{k}(\omega)=\binom{m}{j} \cdot\gauss{s}{r}(\omega).
\]
\item Otherwise $\omega$ is a simple root of $\gauss{n}{k}$. 
\end{itemize}
\end{lemma}

\begin{proof}
Assume $q\geq 1$ is an integer and $\omega\in \Omega_q$. Let $i\geq 1$ be an integer. Observe that $\omega$ is a root of $1-\lambda^i$ if and only if $i$ is a multiple of $q$. In that case, it is a simple root and $1-\lambda^i \sim\frac{i}{\omega}(\omega-\lambda)$ as $\lambda \to \omega$. Assume now that $k\geq 0$ is an integer and consider the Euclidean division $k=j q+r$. We have that 
\[\pi_k(\lambda) \underset{\lambda \to \omega}\sim  \frac{j!q^j}{\omega^j} (\omega-\lambda)^j \left(\prod_{i=1}^{q-1} (1-\omega^i)\right)^j  \prod_{i=1}^{r}  (1-\omega^i) =  \frac{j!q^{2j}}{\omega^j} \pi_r(\omega) (\omega-\lambda)^j.\]
Assume finally that $0\leq k\leq n$ and consider the Euclidean divisions $k=j q+r$ and $n=m q+s$. We have that
\[
\left\lfloor\frac{n}{q}\right\rfloor-1 \leq\left\lfloor\frac{k}{q}\right\rfloor+\left\lfloor\frac{n-k}{q}\right\rfloor \leq\left\lfloor\frac{n}{q}\right\rfloor .
\]
If $r\leq s$, the right side is an equality and 
\[\gauss{n}{k}(\lambda)\underset{\lambda \to \omega}\longrightarrow \frac{\ds \frac{m! q^{2m}}{\omega^m}}{\ds \frac{j! q^{2j}}{\omega^j}  \frac{(m-j)! q^{2(m-j)}}{\omega^{m-j}}}  \frac{\pi_s(\omega)}{\pi_r(\omega) \pi_{s-r}(\omega)}=\binom{m}{j} \gauss{s}{r}(\omega).
\]
Otherwise, the left side is an equality and $\omega$ is a simple root of $\gauss{n}{k}$. 
\end{proof}

Assume now that $m\geq 1$ is an integer,  $(k_1, \ldots, k_m)\in \N^m$ and $n := k_1+\cdots+k_m$. The multinomial coefficient
\[\binom{n}{k_1,\ldots,k_m}:= \frac{n!}{k_1!\cdots k_m!}\]
is an integer: it is equal to the number of ways of distributing $n$ distinct objects into $m$ distinct bins, with $k_1$ objects in the first bin, $k_2$ objects in the second bin, and so on. 
The Gaussian polynomial 
\[\gauss{n}{k_1,\ldots,k_m}:=\frac{\pi_n}{\pi_{k_1}\cdots \pi_{k_m}} \in \Z[\lambda]\]
is a polynomial with  integer coefficients. Indeed,
\[\gauss{n}{k_1,\ldots,k_m}= \gauss{k_1+k_2}{k_2} \times\gauss{k_1+k_2+k_3}{k_3} \times \cdots \times \gauss{k_1+k_2+\ldots +k_m}{k_m} \in \Z[\lambda].\]

\subsection{Parabolic multiplicity}

Assume
\[F(z) =z \sum_{k\geq  0} f_k z^k\quad \text{with}\quad f_0:=1\]
is a formal power series fixing the origin with multiplier $1$. 
For each $\lambda \in \C$, define $F_\lambda := \lambda F$. 

For every $n\in \N^\ast$, 
\[F_\lambda ^{\circ n}(z)  = \lambda^n z \sum_{k\geq  0} c_{n,k}(\lambda) z^k\quad \text{with}\quad c_{n,0} = 1\quad\text{and}\quad c_{n,k}(\lambda)\in \Z[f_1,\ldots,f_k][\lambda].\]
If $\omega\in \Omega_n$ is a primitive $n$-th root of unity, then $F_\omega^{\circ n}$ is tangent to the identity at $0$. 
Suppose that $F_\omega^{\circ n}$ is not the identity and let $m\geq 1$ be the smallest integer such that 
\[F_\omega^{\circ n}(z)  = z\left(1 + cz^m + \O(z^{m+1})\right)\quad \text{with}\quad c\neq 0.\]
Then, 
\[F_\omega^{\circ n}\circ F_\omega(z) = F_\omega(z)\left(1 + c\omega^m z^m + \O(z^{m+1})\right)\]
and 
\begin{eqnarray*}
F_\omega\circ F_\omega^{\circ n}(z) &=& \omega z\left(1 + cz^m + \O(z^{m+1})\right)\left(\sum_{k\geq  0} f_k z^k\left(1 + cz^m +\O(z^{m+1})\right)^k\right)\\
&=&  F_\omega(z)\left(1 + c z^m + \O(z^{m+1})\right).
\end{eqnarray*}
Since $F_\omega$ commutes with $F_\omega^{\circ n}$, it follows that $\omega^m=1$, so that $m$ is a multiple of $n$. 
Therefore, if $F_\omega^{\circ n}$ is not the identity, then 
\[F_\omega^{\circ n}(z)  = z\left(1 + cz^{\nu n} + \O(z^{\nu n +1})\right)\]
for some integer $\nu\geq 1$ called the parabolic multiplicity of $F_\omega$ at $0$.

We are interested in studying whether the parabolic multiplicity of $F_\omega$ at $0$ is equal to $1$, i.e., whether $c_n(\omega)\neq 0$, where $c_n$ is the polynomial defined by 
\[c_n:=c_{n,n}.\]
We use this notation throughout the paper.

\subsection{Linearization\label{sec:lin}}

When $\lambda$ is neither $0$ nor a root of unity, there exists a unique formal power series
\[G_\lambda(z) = z \sum_{k\geq 0} g_k(\lambda) z^k \quad \text{with}\quad g_0:=1,\]
such that
\[G_\lambda (\lambda z) = F_\lambda \circ G_\lambda(z).\]
The coefficients $g_k$ satisfy the recursive relation
\[g_0=1\quad \text{and}\quad \forall k\geq 1,\quad g_k = \frac{1}{\lambda^k-1}\sum_{m=1}^k f_m \cdot \left(\sum_{(j_0,\ldots, j_m)\in \N^{m+1}\atop j_0+\ldots+j_m = k-m} g_{j_0}\cdots g_{j_m}\right).\]
In particular, for $k\geq 1$, the coefficient $g_k\in \Q[f_1,\ldots,f_k](\lambda)$ is a rational function whose poles are contained in the set of $j$-th roots of unity with $j\in \ob1;k\cb$. 

For $k\geq 0$, set $a_k := \pi_k g_k$. Then, 
\begin{equation}\label{eq:lin}
a_0=1\quad \text{and}\quad  \forall k\geq 1,\quad a_k = - \sum_{m=1}^k f_m \frac{\pi_{k-1}}{\pi_{k-m}} \left(\sum_{(j_0,\ldots, j_m)\in \N^{m+1}\atop j_0+\ldots+j_m = k-m} \gauss{k-m}{j_0,\ldots,j_m} a_{j_0}\cdots a_{j_m}\right).
\end{equation}
In particular,  $a_k$ is a polynomial with coefficients in $\Z[f_1,\ldots, f_k]$. 

\begin{lemma}\label{lem:aqcq}
For every integer $n\geq 1$,
\[ \quad a_n+c_n \equiv 0 \pmod{\Phi_n}.\]
\end{lemma}

\begin{proof}
Let $n\geq 1$ be an integer and fix $\omega\in \Omega_n$. It suffices to show that $a_n(\omega) +c_n(\omega)=0$.  
On the one hand, we have 
\[ G_\lambda\left( \lambda^n z \right) = \lambda^n z \sum_{k \geq 0} g_k(\lambda) \lambda^{k n} z^k.\] 
The coefficient of $z^{n +1}$ in $G_\lambda(\lambda^n z)$ is  $\lambda^{n +n^2} g_n(\lambda)$. 
On the other hand,
 \[ F_\lambda^{\circ n} \circ G_\lambda(z) = \lambda^n \sum_{k \geq 0} c_{n, k}(\lambda) \bigl(G_\lambda(z)\bigr)^{k +1} 
 = \lambda^n \sum_{k \geq 0} c_{n, k}(\lambda) z^{k +1} \left( \sum_{j\geq 0} g_j(\lambda) z^j\right)^{k +1} ,\] 
 with $c_{n, 0} = 1$ and $g_0= 1$. 
 Therefore, the coefficient of $z^{n +1}$ in $F_\lambda^{\circ n} \circ G_\lambda(z)$ is 
 \[\lambda^n \sum_{k = 0}^n c_{n, k}(\lambda)\phi_{n, k}(\lambda)\quad \text{with}\quad \forall k \in\ob0;n\cb,
 \quad \phi_{n, k} := \sum_{\substack{( r_0, \ldots, r_k) \in \N^{k +1}\atop r_0+\cdots +r_k = n -k}} g_{r_0} \cdots g_{r_k}.\] 
 Consequently, since $G_\lambda( \lambda^n z ) = F_\lambda^{\circ n} \circ G_\lambda(z)$,
 \[ \lambda^{n^2} g_n = \sum_{k = 0}^n c_{n, k} \phi_{n, k}\quad \text{which yields}\quad 
 \left(1-\lambda^{n^2}\right) g_n +c_n+\sum_{k = 1}^{n -1} c_{n, k}\phi_{n, k}=0. \] 
 Now, note that for $k \in \ob 1;n -1\cb$, the rational function  $\phi_{n, k}$ belongs to $\Q[ g_1, \cdots, g_{n -1} ]$, and in particular has no pole at $\omega$. In addition,  $c_{n, k}(\omega) = 0$ for each $k \in \ob 1;n -1\cb$. It follows that 
 \[ \lim_{\lambda \rightarrow \omega} \left(1-\lambda^{n^2}\right) g_n(\lambda) +c_n(\omega) =0.\] 
 Thus, it suffices to show that $\lim\limits_{\lambda \rightarrow \omega} \left(1-\lambda^{n^2}\right) g_n(\lambda) = a_n(\omega)$. 
 We have that
 \[\left(1-\lambda^{n^2}\right) g_n(\lambda) = a_n(\lambda) \frac{1-(\lambda^n)^n}{1-\lambda^n}  \frac{1}{\ds \prod_{k = 1}^{n -1} \left( 1- \lambda^k\right)} \underset{\lambda\to \omega}\longrightarrow a_n(\omega),\]
 because $\ds \frac{1-x^n}{1-x}\underset{x\to 1}\longrightarrow n$ and $\ds \prod_{k = 1}^{n -1} ( 1-\omega^k )=n$. 
 \end{proof}

Consequently, 
\[\forall n\geq 1,\quad \forall \omega \in \Omega_n,\quad c_n(\omega) = 0\quad \Longleftrightarrow \quad a_n(\omega) = 0.\]

\section{The exponential family}

We now consider the family $\{F_\lambda\}_{\lambda\in \C^\ast}$ of transcendental entire maps defined by 
\[F_\lambda(z) := \lambda z \e^{-z}.\]
Given an integer $n\geq 1$, let $c_n\in \Q[\lambda]$ be the polynomial  defined by 
\[c_n:= c_{n,n}\quad \text{with}\quad F_\lambda^{\circ n}(z) = \lambda^n z \sum_{k\geq 0} c_{n,k}(\lambda) z^k.\]
Corollary~\ref{cor:exp} asserts that $c_n$ does not vanish on $\Omega_n$. We now present  an arithmetic proof of this result. 
To this end, consider the family of linearizing maps $G_\lambda$ defined in \S\ref{sec:lin}: 
\[G_\lambda(z) = z \sum_{k\geq 0} g_k(\lambda) z^k \quad \text{with}\quad g_0:=1\quad \text{and}\quad G_\lambda (\lambda z) = F_\lambda \circ G_\lambda(z).\]
It is convenient to express the linearizing map $G_\lambda$ as 
\[G_\lambda(z) = z  \exp \bigl( H_\lambda(z)\bigr)\quad \text{with}\quad H_\lambda(z) = \sum_{k\geq 1} h_k(\lambda) z^k.\]

\begin{lemma}\label{lem:gkhk}
For every integer $k\geq 1$,  
\[g_{k} -h_{k} \in \Q\left[ h_{1}, \dotsc, h_{k -1} \right].\]
\end{lemma}

\begin{proof}
Observe that 
\[ G_\lambda(z) = z \exp\left( \sum_{k \geq 1} h_{k}(\lambda) z^{k} \right) = z \prod_{k \geq 1} \exp\left( h_{k}(\lambda) z^{k} \right) . \] 
Therefore,
\[ G_\lambda(z) = z \prod_{k \geq 1} \left( \sum_{r_{k} \geq 0} \frac{h_{k}(\lambda)^{r_{k}}}{r_{k}!} z^{k r_{k}} \right)  = z \sum_{k \geq 0} g_{k}(\lambda) z^k\]
with $g_0=1$ and for every $k\geq 1$, 
\[g_k =  \sum_{( r_1, \ldots, r_k ) \in \N^k\atop \sum j r_j = k} \frac{h_{1}^{r_{1}} \cdots h_{k}^{r_{k}}}{r_{1}! \cdots r_{k}!} =  h_{k} +\sum_{\left( r_{1}, \ldots, r_{k-1} \right) \in \N^{k-1}\atop \sum j r_j = k} 
\frac{h_{1}^{r_{1}} \cdots h_{k -1}^{r_{k -1}}}{r_{1}! \cdots r_{k -1}!}.\qedhere
\]
\end{proof}

It follows that $h_k(\lambda)$ is a rational function whose poles are contained in the set of $j$-th roots of unity with $j\in \ob 1;k\cb$. 
We have in fact the following stronger result. For $k\geq 1$, set 
\[b_k :=  (k-1)! \pi_k h_k.\]

\begin{prop}\label{prop:recursiveexp}
We have that
\[b_1 = 1\quad \text{and}\quad \forall k\geq 1,\quad  b_{k+1} =  \sum_{j=1}^k  j   \binom{k-1}{j-1} \gauss{k}{j} b_j b_{k+1-j}.\]
In particular,  each polynomial $b_k\in \Z[\lambda]$ has integer coefficients.
\end{prop}

\begin{proof}
Define
\[A_\lambda(z) = \frac{H_\lambda(\lambda z)- H_\lambda(z)}{z} = \sum_{k\geq 1} (\lambda^k-1) h_k(\lambda) z^{k-1}.\]
The linearizing equation 
$G_\lambda (\lambda z) = F_\lambda \circ G_\lambda(z)$ yields
\[\lambda z \exp \bigl( H_\lambda(\lambda z)\bigr) = \lambda z \exp \bigl( H_\lambda(z) \bigr) \exp\bigl(-z \exp \circ H_\lambda(z)\bigr),\]
so that 
\[H_\lambda(\lambda z) = H_\lambda(z) - z \exp \circ H_\lambda(z)\]
that is, 
\begin{equation}\label{eq:bexpa}
A_\lambda(z) = -\exp \bigl( H_\lambda(z)\bigr).
\end{equation}
Note that 
\[(\lambda-1)h_1(\lambda) = A_\lambda(0) = - \exp \bigl(H_\lambda(0)\bigr)=-1.\]
Thus, 
\[h_1(\lambda) = \frac{1}{1-\lambda} \quad \text{and}\quad b_1(\lambda) = 1.\]
In addition, differentiating Equation~\eqref{eq:bexpa} with respect to $z$ yields
\[A_\lambda'(z) = - H_\lambda'(z)  \exp \bigl( H_\lambda(z) \bigr)=H_\lambda'(z) A_\lambda(z).\]
It follows that
\[\sum_{k\geq 1} (\lambda^{k+1}-1) k h_{k+1}(\lambda) z^{k-1} = \left(\sum_{k\geq 1} k h_k(\lambda) z^{k-1}\right)\left(\sum_{k\geq 1} (\lambda^k-1)h_k(\lambda) z^{k-1}\right).\]
Hence, 
\[\forall k\geq 1,\quad (\lambda^{k+1}-1) kh_{k+1}(\lambda)  = \sum_{j=1}^k j h_j(\lambda)  (\lambda^{k+1-j}-1)h_{k+1-j}(\lambda),\]
so that 
\[
\forall k\geq 1,\quad  \frac{b_{k+1}}{(k-1)! \pi_k} = \sum_{j=1}^k j \frac{b_j}{(j-1)! \pi_j}  \frac{b_{k+1-j}}{(k-j)! \pi_{k-j}} .\]
As a consequence, 
\[\forall k\geq 1,\quad b_{k+1} =  \sum_{j=1}^k  j   \binom{k-1}{j-1} \gauss{k}{j} b_j b_{k+1-j}.\qedhere\]
\end{proof}

Fix an integer $n\geq 1$ and define 
\[\beta_n := \resultant(\Phi_n,b_n) .\]
Since $b_n\in \Z[\lambda]$ has integer coefficients,  $\beta_n\in \Z$ is an integer. The following proposition implies that 
\[\beta_n = \resultant(\Phi_n,-(n-1)!c_n).\]

\begin{prop}\label{prop:betac}
For every integer $n\geq 1$,
\[b_n + (n-1)!c_n  \equiv 0 \pmod{\Phi_n}.\]
\end{prop}

\begin{proof}
Fix $n\geq 1$ and $\omega \in \Omega_n$. It suffices to prove that  $b_n (\omega)+ (n-1)!c_n (\omega) = 0$. 
Define $a_n := \pi_n g_n$. By Lemma~\ref{lem:aqcq},  $a_n (\omega)+ c_n (\omega) = 0$. Thus, it is enough to show that $b_n(\omega) = (n-1)! a_n(\omega)$. 
From
\[ a_n= \pi_n g_n\quad \text{and}\quad b_n = (n-1)! \pi_n h_n,\]
we deduce that 
\[b_n(\omega) - (n-1)!a_n(\omega) =(n-1)!  \lim_{\lambda\to \omega} \pi_n(\lambda) \bigl(h_n(\lambda) - g_n(\lambda)\bigr).\]
By Lemma~\ref{lem:gkhk},
\[h_n-g_n \in \Q[h_1,\ldots,h_{n-1}].\]
The rational maps $h_1,\ldots,h_{n-1}$ are holomorphic near $\omega$ and $\pi_n(\lambda)\to  0$ as $\lambda\to \omega$. It follows that $b_n(\omega) = (n-1)!a_n(\omega)$ as required. 
\end{proof}

As a consequence, the polynomial $c_n$ does not vanish on $\Omega_n$ if and only if $b_n$ does not vanish on $\Omega_n$. 
In order to prove this latter result,  we will follow an approach of Broer, Sim\'o and Tatjer \cite[Proposition B.1]{bst}. 

\begin{prop}\label{prop:bst}
For every integer $k\geq 2$,
\[b_{k+1}  \equiv \frac{\pi_k}{\pi_1^k} \pmod{k}.\]
\end{prop}

\begin{proof}
%We keep the notation introduced in the proof of Proposition~\ref{prop:recursiveexp}. 
By Equation~\eqref{eq:bexpa},
\[\sum_{k\geq 0} (\lambda^{k+1}-1) h_{k+1}(\lambda) z^k = -\exp\left( \sum_{k\geq 1} h_k(\lambda) z^k\right)  = -\prod_{k\geq 1} \exp \bigl( h_k(\lambda) z^k\bigr).\]
It follows that 
\[\forall k\geq 0,\quad (\lambda^{k+1}-1) h_{k+1}(\lambda) = -\sum_{(r_1,\ldots,r_m)\atop \sum jr_j = k} \prod_{j=1}^{m} \frac{h_j^{r_j}(\lambda)}{r_j!},\]
where the sum is taken over all finite sequences $(r_1,\ldots r_m)$ such that $\ds \sum_{j=1}^m j r_j = k$. 
This may be rewritten as 
\[\forall k\geq 0,\quad b_{k+1} = \sum_{(r_1,\ldots,r_m)\atop \sum jr_j = k} \underset{\text{multinomial coefficient}}{\underbrace{\frac{k!}{(1!)^{r_2}\cdots \bigl((m-1)!\bigr)^{r_m}  r_1!\cdots r_m!}}} \underset{\text{Gaussian polynomial}}{\underbrace{\frac{\pi_k}{\pi_1^{r_1}\cdots \pi_m^{r_m}}}}b_1^{r_1}\cdots b_m^{r_m}.\]
The term corresponding to $(r_1,0,\ldots,0)$ yields the contribution $\ds \frac{\pi_{k}}{\pi_1^{k}}$. We claim that the other multinomial coefficients are divisible by $k$. Indeed, in that case set $r:=r_2+\cdots + r_m\in \ob 1;k\cb$. Hence, 
\[ \frac{k!}{(1!)^{r_2}\cdots \bigl((m-1)!\bigr)^{r_m} \ r_1!\cdots r_m!} =\underset{\text{divisible by }k}{\underbrace{\frac{k!}{(k-r)!}}} \frac{1}{r_2!\cdots r_m!} \binom{k-r}{\underset{r_2}{\underbrace{1,\ldots,1}},\ldots,\underset{r_m}{\underbrace{m-1,\ldots,m-1}},r_1}.\]
This last multinomial coefficient is equal to the number of ways of distributing $k-r$ distinct objects in:
\begin{itemize}
\item $r_2$ bins containing $1$ object,
\item \ldots 
\item $r_m$ bins containing $m-1$ objects and
\item $1$ bin containing $r_1$ objects. 
\end{itemize}
Since we can permute the $r_2$ bins containing $1$ object, \ldots, and the $r_m$ bins containing $m-1$ objects, the last multinomial coefficient  is divisible by $r_2!\cdots r_m!$.
\end{proof}

\begin{proof}[Arithmetic proof of Corollary~\ref{cor:exp}]
Let $n\geq 3$ be an integer. 
Observe that $\ds \frac{\pi_{n-1}}{\pi_1^{n-1}}$ is a product of cyclotomic polynomials $\Phi_j$ with $j\in \ob1;n-1\cb$. For such an integer $j$, the resultant of $\Phi_j$ and $\Phi_n$ is either equal to $1$ or a power of a prime divisor of $n$. Since $n$ and $n-1$ are coprime, 
\[\beta_n \equiv \resultant\left( \Phi_n,\frac{\pi_{n-1}}{\pi_1^{n-1}}\right) \not\equiv 0 \pmod{n-1}.\]
Hence, for every $n\geq 3$, the polynomial $b_n$, and thus the polynomial $c_n$, does not vanish on $\Omega_n$. 
Since $c_1 = -1$ and $c_2 = -1$ neither vanish on $\Omega_n$, this completes the arithmetic proof of Corollary~\ref{cor:exp}. 
\end{proof}

\begin{proof}[Proof of Proposition~\ref{prop:beta}]
A simple computation yields $\beta_1 = \beta_2 = 1$. Let $n\geq 3$ be an integer. As noted above, $\beta_n$ is an integer. 
In addition, 
\[\resultant( \Phi_n,\pi_1)  =  \begin{cases} 1&\text{if }n \text{ is not a prime power},\\
p&\text{if } n\text{ is a power of the prime }p\end{cases}  \quad \text{and}\quad  \resultant(\Phi_n,\pi_{n-1}) = n^{\varphi(n)}.\]
It follows that when $n$ is not a prime power, then  $\beta_n\equiv 1 \pmod{n-1}$ and if $n$ is a power of the prime $p$, then $p^{n-1}\beta_n\equiv 1 \pmod{n-1}$.
\end{proof}

\section{Unicritical polynomials}

In this section, $d\geq 2$ is an integer. Consider the family $\{F_\lambda\}_{\lambda\in \C^\ast}$ of unicritical polynomials of degree $d$ defined by 
\[F_\lambda(z) := \frac\lambda{d^2} \bigl(1- (1-dz)^d\bigr).\]
Given $n\in \N^\ast$, let  $c_n\in \Z[\lambda]$ be the polynomial with integer coefficients defined by 
\[c_n:= c_{n,n}\quad \text{with}\quad F_\lambda^{\circ n}(z) = \lambda^n z \sum_{k\geq 0} c_{n,k}(\lambda) z^k.\]
It follows from Theorem~\ref{theo:fatou} that $c_n$ does not vanish on $\Omega_n$. We will now give an arithmetic  proof of this result. 

For this purpose, consider the family of linearizing maps $G_\lambda$ defined in \S\ref{sec:lin}: 
\[G_\lambda(z) = z \sum_{k\geq 0} g_k(\lambda) z^k \quad \text{with}\quad g_0:=1\quad \text{and}\quad G_\lambda (\lambda z) = F_\lambda \circ G_\lambda(z).\]
By \S\ref{sec:lin}, since the coefficients of the polynomial $F_1(z)$  are integers, 
\[a_n:= \pi_n g_n\in \Z[\lambda].\]  
In addition, by Lemma~\ref{lem:aqcq}, 
\[\forall n\geq 1,\quad a_n + c_n \equiv 0\pmod{\Phi_n}.\]
Therefore, Corollary~\ref{cor:quad} is equivalent to proving that for $n\geq 1$, the polynomial $a_n$ does not vanish on $\Omega_n$, i.e., that $a_n$ and $\Phi_n$ are coprime. 
The proof relies on the Bang-Zsigmondy Theorem \cite{ba,z}, and we first need to deal with exceptional cases $n\in \{1,2\}$ for arbitrary degrees,  or $n=6$ for $d=2$. 

\subsection{The case $n=1$}

Observe that 
\[F_1(z) = z - \frac{d(d-1)}{2} z^2+\O(z^3) .\]
It follows that 
\[c_1 = - \frac{d(d-1)}{2} \neq 0.\]

\subsection{The case $n=2$}

Observe that 
\[F_{-1}^{\circ 2}(z) = z - \frac{d^2(d-1)(d+1)}{6} z^3 + \O(z^4) .\]
It follows that 
\[c_2 =  - \frac{d^2(d-1)(d+1)}{6}\neq 0.\]

\subsection{The case $d=2$ and $n=6$}

Let us temporarily assume that $d=2$, so that $F_\lambda(z) = \lambda z(1-z)$. Further assume that $n=6$. A computer-assisted computation yields
\[\resultant(\Phi_6,c_6) = 10128 = 2^4\cdot 3\cdot 211.\]
It follows that $c_6$ and $\Phi_6$ are coprime.

\subsection{A $p$-adic proof that $a_n$ and $\Phi_n$ are coprime\label{sec:padic}}

Throughout this proof,  $d\geq 2$ and $n\geq 3$ are fixed integers. If $d=2$, then we assume that $n\neq 6$. 

By the Bang-Zsigmondy Theorem, $d^n-1$ has a prime divisor $p$ which does not divide any $d^k-1$ for $k<n$. Let $p$ be such a prime number and set $\F_p := \Z/p\Z$. 
Since $d^k \not\equiv 1\pmod{p}$ for $k<n$ and since $d^n \equiv 1 \pmod{p}$, the order of $d$ in $\F_p^\times$ is equal to $n$, so that $n$ divides the cardinality of $\F_p^\times$, namely $p-1$. In particular, $p\geq n+1$. 

Equip $\Q$ with its $p$-adic norm $|\cdot|_p$ and let $\Qp$ be its completion. 
Since $\Phi_n(\lambda)\in \Z[\lambda]$ is monic and irreducible over $\Q$ and since $a_n(\lambda)\in \Z[\lambda]$, if $a_n$ and $\Phi_n$ have a common zero in $\C$, then $\Phi_n$ divides $a_n$ in $\Z[\lambda]$, whence every root of $\Phi_n$ in $\Qp$ is also a root of $a_n$. 

By assumption, 
\[\bigl|\pi_{n-1}(d)\bigr|_p =1 ,\quad \bigl|\pi_n(d)\bigr|_p  = |d^n-1|_p<1 \quad \text{and}\quad \bigl|\pi_n'(d)\bigr|_p =1.\]
Indeed, $\pi_n$ has simple roots in $\overline \F_p$ since $n$ does not divide $p$. If we had $\bigl|\pi_n'(d)\bigr|_p<1$, then $d$ would be a multiple root of $\lambda^n-1\in \F_p[\lambda]$, thus a root of $(\lambda^n-1)' = n\lambda^{n-1}\in \F_p[\lambda]$; however $p$ does not divide $nd^{n-1}$. 

\begin{remark}
It seems that we usually have $|d^n-1|_p = \frac{1}{p}$. However, we may have $|d^n-1|_p\leq \frac{1}{p^2}$. This is the case when $d=2$, $p=3511$ is the second Wieferich prime, and $n=1755$ is the order or $2$ in $(\Z/p\Z)^\times$:
\[ v_{3511}(2^{1755}-1) =2.\]
\end{remark}

By the Hensel Lemma, $\pi_n$ has a unique root $\omega\in \Qp$ in the open ball of radius $1$ centered at $d$, and 
\[\left|\omega - d + \frac{\pi_n(d)}{\pi'_n(d)} \right|_p \leq  \left|\frac{\pi_n(d)}{\pi'_n(d)} \right|_p^2 = |d^n-1|_p^2 .\]
In particular, 
\[|\omega-d|_p = |d^n-1|_p\quad \text{and}\quad |\omega|_p=1.\]
In addition, since $\pi_{n-1}\in \Z[\lambda]$ and $|\omega|_p=1$, 
\[\bigl|\pi_{n-1}(\omega) - \pi_{n-1}(d)\bigr|_p \leq |\omega- d|_p<1, \quad \text{so that}\quad \bigl|\pi_{n-1}(\omega)\bigr|_p = \bigl|\pi_{n-1}(d)\bigr|_p=1.\]
In particular, $\pi_{n-1}(\omega)\neq 0$, whence $\Phi_n(\omega) = 0$. It is therefore sufficient to prove that $a_n(\omega)\neq 0$.

The polynomial
\[P_n(\lambda) := a_n(\lambda) - a_n(d) - a_n'(d) (\lambda-d)\in \Z[\lambda]\] 
satisfies $P_n(d) = 0$ and $P_n'(d) = 0$, so that $P_n(\lambda) \in (\lambda-d)^2\Z[\lambda]$.  Since $|\omega|_p = 1$, it follows  that  
\[\bigl|P_n(\omega)\bigr|_p \leq |\omega-d|_p^2 =  |d^n-1|_p^2.\]
Thus, 
\[\left| a_n(\omega) - a_n(d) + a_n'(d)  \frac{\pi_n(d)}{\pi'_n(d)} \right|_p \leq   |d^n-1|_p^2 .\]

By definition,  $a_n = \pi_n g_n$. It follows that $a'_n = \pi_n' g_n + \pi_n g'_n$, so that 
\[a_n(d) - a_n'(d)  \frac{\pi_n(d)}{\pi'_n(d)}  = - \frac{\pi_n^2(d)}{\pi_n'(d)} g_n'(d).\]

\begin{lemma}
We have that
\[ \bigl|g_n'(d)\bigr|_p  \geq \frac{1}{|d^n-1|_p}.\]
\end{lemma}

\begin{proof}
Recall that $G_\lambda(z)$ is the unique formal power series which satisfies  
\[G_\lambda(\lambda z) = F_\lambda \circ G_\lambda(z)\quad \text{and}\quad  G_\lambda(z) = z \sum_{k \geq  1} g_k(\lambda) z^k \quad \text{with} \quad g_0(\lambda) = 1. \]
In addition, the polynomial $F_d$ fixes the critical point $\frac{1}{d}$, thus is conjugate to  $z^d$, and we have  
\[ G_d(z) = \frac{1 -\e^{-d z}}{d} , \quad \text{so that} \quad \forall k \geq 0, \quad g_k(d) = \frac{(-d)^k}{(k +1)!}. \]
In order to compute $g_k'(d)$ for $k\geq 1$, let us study the variation of $G_\lambda$ as $\lambda$ varies in a neighborhood of $d$. 
Let $K_\lambda$ and $H_\lambda$ be defined by: 
\[G_d\circ K_\lambda = F_\lambda \circ G_d\quad \text{and}\quad G_\lambda =  G_d\circ H_\lambda ,\]
so that $G_d$ conjugates $K_\lambda$ to $F_\lambda$ and $H_\lambda$ linearizes $K_\lambda$: 
\[\xymatrix@C=2cm@R=2cm{
\dto_{z\mapsto \lambda z} \rto_{H_\lambda} \ar@/^1pc/[rr]^{G_\lambda}& \dto^{K_\lambda} \rto_{G_d} & \dto^{F_\lambda} \\ 
\rto^{H_\lambda} \ar@/_1pc/[rr]_{G_\lambda} & \rto^{G_d} & 
}\]
Therefore,
\[K_{d+\eps} (z)  = d \bigl( z + \eps\kappa(z) + \O(\eps^2)\bigr) \quad \text{with}\quad \kappa(z)\in z\Q[[z]].\]
Note that 
\[
G_d\circ K_{d+\eps}(z)   =  G_d\bigl(dz +\eps d \kappa(z) + \O(\eps^2)\bigr)
 =G_d(dz) + \eps d G_d'(dz) \kappa(z) + \O(\eps^2),
\]
and 
\[ F_{d +\eps} \circ G_d(z) =   \frac{d +\eps}{d} F_d \circ G_d(z) = \left(1+\frac{\eps}{d}\right) G_d(dz) .\] 
Equating the terms of order $\eps$ yields 
\[ \kappa(z) = \frac{G_d(dz)}{d^2 G_d'(dz)} = \frac{\e^{d^2 z} -1}{d^3} = z\sum_{k\geq 0} \kappa_k z^k\quad \text{with}\quad \kappa_k =  \frac{d^{2k-1}}{(k+1)!} .\]
Observe that 
\[\forall k\in \ob0;n-1\cb \quad |\kappa_k|_p = 1\quad \text{and}\quad |\kappa_n|_p\geq 1.\]
Let us now identify $\eta(z) \in z^2\Q[[z]]$ such that  
\[ H_{d+\eps}(z)  = z +\eps \eta(z) +\O(\eps^2) .\] 
Equating the terms of order $\varepsilon$ in the equality
\[H_{d+\eps}\bigl((d+\eps)z\bigr) = K_{d+\eps}\circ H_{d+\eps}(z),\]
we easily get the relation 
\[ z+ \eta(d z) = d \eta(z) +d \kappa(z).\] 

\begin{remark}
This equation expresses the fact that if $M_d$ is the linear map $z\mapsto dz$, then we have the following equality of vector fields:
\[\frac{z}{d} \partial_z + M_d^*\bigl(\eta(z)\partial_z\bigr) = \eta(z) \partial_z + \kappa(z)\partial_z.\] 
\end{remark}

Writing
\[\eta(z)  = z\sum_{k\geq 1}  \eta_k z^k,\]
we deduce that
\[\forall k\geq 1,\quad d^{k+1} \eta_k-d\eta_k =d \kappa_k \quad\text{so that}\quad \eta_k = \frac{\kappa_k}{d^k-1}.\]
Observe that 
\[\forall k\in \ob0;n-1\cb \quad  |\eta_k|_p =  |\kappa_k|_p = 1\quad \text{and}\quad |\eta_n|_p= \frac{ |\kappa_n|_p}{|d^n-1|_p}\geq \frac{1}{|d^n-1|_p}.\]
Finally, since $G_d'(z) = \e^{-d z}$ and 
$H_{d+\eps}(z) = z +\eps \eta(z) +\O(\eps^2)$, 
\[ G_{d +\eps}(z) = G_d \circ H_{d+\eps}(z) = G_d(z) +\eps \e^{-d z} \eta(z) +\O( \eps^2 ), \] 
which yields
\[ z \sum_{k \geq  1} g_k'(d) z^{k} = \e^{-d z} \eta(z) .\] 
As a consequence, equating the coefficients yields
\[ \forall k \geq 1, \quad g_k'(d) =\sum_{j=1}^k \frac{ (-d)^{k-j}}{(k-j)!}\eta_j.\]
Note that if $j<n<p$, then 
\[\left|\frac{ (-d)^{n-j}}{(n-j)!}\eta_j\right|_p = |\eta_j|_p = 1.\]
In addition, 
\[\left|\frac{ (-d)^{n-n}}{(n-n)!}\eta_n\right|_p = |\eta_n|_p  \geq   \frac{1}{|d^n-1|_p} .\]
Since there is only one term of norm greater than $1$ in the sum, we deduce that 
\[ \bigl|g_n'(d)\bigr|_p =  |\eta_n|_p \geq \frac{1}{|d^n-1|_p}.\qedhere\]
\end{proof}

It follows that
\[ \left|a_n(d) - a_n'(d)  \frac{\pi_n(d)}{\pi'_n(d)} \right|_p  = \left|\frac{\pi_n^2(d)}{\pi_n'(d)} g_n'(d) \right|_p \geq  |d^n-1|_p.\]
As a consequence, 
\[\bigl|a_n(\omega)|_p \geq   |d^n-1|_p.\]
This proves that $a_n(\omega)\neq 0$, which completes our arithmetic proof. 

\begin{remark}
The previous computations show that 
\[\forall k\geq 1,\quad g_k'(d) = \sum_{j=1}^k  \left(\frac{(-1)^{k-j}}{(j+1)!(k-j)! }\  \frac{d^{k+j-1}}{d^j-1}\right).\]
\end{remark}

\begin{remark}
The previous argument shows that $a_n(\omega)$ belongs to the closed ball of radius  $|d^n-1|_p^2$ centered at 
\[\frac{d^{2n-1}}{(n+1)!}\ \frac{\pi_{n-1}(d) \pi_n(d)}{\pi'_n(d)}.\]
\end{remark}

\begin{remark}
Note that 
\[\bigl|a_n(d)\bigr|_p = \bigl|\pi_n(d)g_n(d)\bigr|_p = \left|\frac{d^n-1}{n+1}\right|_p\geq |d^n-1|_p\quad \text{and}\quad 
 \left|a_n'(d)  \frac{\pi_n(d)}{\pi'_n(d)} \right|_p\leq |d^n-1|_p.\]
Consequently,
\[ \bigl|a_n(\omega)\bigr|_p = \left|a_n(d) - a_n'(d)  \frac{\pi_n(d)}{\pi'_n(d)} \right|_p  = \bigl|a_n(d)\bigr|_p.\]
It may happen that $p=n+1$, in which case $\bigl|a_n(\omega)\bigr|_p> |d^n-1|_p$; for example, when $d=2$ and $n\in \{2,4,10,12,18,\ldots\}$ (but not when $n=6$ or $n=16$).
\end{remark}

\section{The quadratic family}

We now restrict our study to the quadratic case $d=2$, i.e., to the family $\{F_\lambda\}_{\lambda\in \C^\ast}$ of quadratic polynomials defined by 
\[F_\lambda(z) := \lambda z (1-z).\]
Fix $n\in \N^\ast$. 
As before, let $c_n\in \Z[\lambda]$ be the polynomial defined by 
\[c_n:= c_{n,n}\quad \text{with}\quad F_\lambda^{\circ n}(z) = \lambda^n z \sum_{k\geq 0} c_{n,k}(\lambda) z^k.\]
We continue to denote by $G_\lambda$  the linearizing maps  defined in \S\ref{sec:lin}: 
\[G_\lambda(z) = z \sum_{k\geq 0} g_k(\lambda) z^k \quad \text{with}\quad g_0:=1\quad \text{and}\quad G_\lambda (\lambda z) = F_\lambda \circ G_\lambda(z).\]
Yoccoz \cite{yoccoz} observed that the polynomials $a_k:= \pi_k g_k \in\Z[\lambda]$ satisfy the recursion relation
\begin{equation}\label{eq:recursion}
a_0=1 \quad \text {and} \quad a_{k+1}=\sum_{j=0}^{k}\gauss{k}{j}a_j a_{k-j}.
\end{equation}
This follows immediately from Equation~\eqref{eq:lin}. The first few polynomials $a_k$ are:
\begin{itemize}
\item $a_1=1$; 
\item $a_2=2$; 
\item $a_3=\lambda+5$;
\item $a_4 = 2(2\lambda^2 + 3\lambda + 7)$;
\item $a_5 = 2(3\lambda^4 + 8\lambda^3 + 14\lambda^2 + 14\lambda + 21)$;
\item $a_6 = 4(\lambda^7 + 8\lambda^6 + 12\lambda^5 + 27\lambda^4 + 32\lambda^3 + 37\lambda^2 + 30\lambda + 33)$.
%\item $a_7 = \lambda^{11} + 27\lambda^{10} + 89\lambda^9 + 204\lambda^8 + 327\lambda^7 + 536\lambda^6 + 660\lambda^5 + 811\lambda^4 + 756\lambda^3 + 705\lambda^2 + 495\lambda + 429$.
\end{itemize}

As observed by Yoccoz, denoting by $k\geq 1$ the unique integer such that $2^k\leq n+1<2^{k+1}$, 
\[\deg a_n = \frac{n(n+1)}{2}- (k+1)(n+1) + 2^{k+1} -1\]
 and the leading coefficient is $\ds \binom{2^k}{n+1-2^k}$.

\subsection{Resultants}

Set
\[\alpha_n:=\resultant(\Phi_n,a_n).\]
By Lemma~\ref{lem:aqcq}, $a_n+c_n \equiv 0\pmod{\Phi_n}$, so that 
\[\alpha_n:=\resultant(\Phi_n,-c_n).\]
The first values of $\alpha_n$ are listed in Table~\ref{table}.

\begin{table}[htbp]
\begin{tabular}{|c|l|c|}
\hline$n$ & $\alpha_n$ & $\frac{1}{n} \alpha_n \pmod{n}$ \\
\hline 1 & ${\color{blue} 1}$ & 0 \\
\hline 2 & ${\color{blue} 2} $ & 1 \\
\hline 3 & ${\color{blue} 3} \cdot 7$ & 1 \\
\hline 4 & ${\color{blue} 4} \cdot 2\cdot 17$ & 2 \\
\hline 5 & ${\color{blue} 5} \cdot 2^4\cdot 11\cdot 31^2$ & 1 \\
\hline 6 & ${\color{blue} 6} \cdot 2^3\cdot 211$ & 2 \\
\hline 7 & ${\color{blue} 7} \cdot 43^2\cdot 127^2\cdot 987211$ & 1 \\
\hline 8 & ${\color{blue} 8} \cdot 2^6\cdot 17^2\cdot 257\cdot 12073$ & 0 \\
\hline 9 & ${\color{blue} 9} \cdot 2^6\cdot 3\cdot 19\cdot 73^2\cdot 1362336515767$ & 3 \\
\hline 10 & ${\color{blue} 10} \cdot 2^{11}\cdot 31\cdot 41\cdot 11261\cdot 33311$ & 8 \\
\hline 11 & ${\color{blue} 11} \cdot 2^{10}\cdot 23\cdot 89^2\cdot 683^2\cdot 233113\cdot 224105472053\cdot 510582069251$ & 1 \\
\hline 12 & ${\color{blue} 12} \cdot 26429903762132911872$ & 0 \\
\hline 13 & ${\color{blue} 13} \cdot 2^{24}\cdot 3^3\cdot 2731^2\cdot 8191^2\cdot 207914839156014043\cdot 35904310679791841984399$ & 1 \\
\hline
\end{tabular}
\caption{The resultants $\alpha_n$ for $n \in \ob 1;13 \cb$.\label{table}}
\end{table}

A first observation is that $n$ divides $\alpha_n$ for every integer $n\geq 1$.  The following result is proved in \S\ref{sec:valuation} as a corollary of Proposition~\ref{prop:valuation} below.

\begin{prop}\label{prop:divn}
 For all $n \in \N^\ast$,  $\alpha_n \equiv 0 \pmod{n}$.
 \end{prop}

Given an integer $k\in \Z$ and  a prime number $p\geq 2$, let $v_p(k)$ be the $p$-adic valuation of $k$, i.e., the largest  $\nu\in \N\cup \{\infty\}$ such that $p^{-\nu} k\in \Z$. Likewise, given $P\in \Z[\lambda]$, let $v_p(P)$ be the largest $\nu\in \N\cup \{\infty\}$ such that $p^{-\nu} P\in \Z[\lambda]$. 
%\begin{conjecture}
%Assume  $n\geq 2$ is an integer and $p\geq 2$ is a prime factor of $\alpha_n$. Then, 
%\begin{itemize}
%\item either $p$ divides $n$; 
%\item or $p>n$ and $p \equiv 1\pmod{n}$; 
%\item or $p<n$ and $p^{v_p(\alpha_n)} \equiv 1 \pmod{n}$. 
%\end{itemize}
%\end{conjecture}
Recall that $\varphi$ is the Euler totient function, so that the degree of $\Phi_n$ is $\varphi(n)$.

\begin{prop}\label{prop:valuationi}
For  $n\in \N^\ast$, if $p$ is a prime factor of $n$, then $\ds v_p(\alpha_n) \geq  \frac{\varphi(n)}{p-1}$. 
\end{prop}

Here is a list of conjectures suggested by our computations of the resultants $\alpha_n$ for $n\in \ob1;243\cb$. 
Our first observation is that in many cases, if $n$ is a multiple of a prime $p\geq 3$, then  
\[v_p(\alpha_n) = \frac{\varphi(n)}{p-1}.\]
However, there are exceptions. 

\begin{table}[htbp]
\begin{tabular}{|c|l|}
\hline$p$ & $m$ \\
\hline 3 & 5;\ 13;\ 15;\ 20;\ 39;\ 45;\ 60 \\
\hline 5 & 24  \\
\hline 7 &  3;\ 16; \ 21\\
\hline 11 & 2;\ 5;\ 22 \\
\hline 13 &  $12$ \\
\hline 17 &  4;\ 8 \\
\hline 19 & 9\\
\hline 31 & 5 \\
\hline
\end{tabular}
\caption{List of values of $p$ and $m$ with $mp\leq 243$, for which $v_p(\alpha_{mp}) \neq \dfrac{\varphi(mp)}{p-1}$.\label{exceptions}}
\end{table}

It seems that those exceptions do not occur when $n$ is power of $p$.

\begin{conjecture}\label{conj:resprimepower}
If $p\geq 3$ is a prime number and $n = p^r$ for some integer $r\geq 1$, then 
\[v_p(\alpha_n) = \ds \frac{\varphi(n)}{p-1} = p^{r-1}\quad \text{and}\quad \frac{\alpha_n}{p^{v_p(\alpha_n)}} \equiv 1 \pmod{n}.\] 
\end{conjecture}

We prove  Conjecture~\ref{conj:resprimepower} when $r=1$ (see \S\ref{sec:prime}).  
The situation for $p=2$ is slightly different. This may be due to the fact that we are studying polynomials of degree $2$. 
Set 
\[\nu_n := v_2\bigl(a_n(0)\bigr) .\]
We shall see that $a_n(0)$ is the Catalan number $\frac{(2n)!}{n!(n+1)!}$, so that  
$1+\nu_n$ is the sum of digits in the binary expansion of $n+1$. We prove in \S\ref{sec:hatan} that $v_2(a_n) = \nu_n$, so that the polynomial 
\[\hat a_n := 2^{-\nu_n} a_n\] has integer coefficients. 
Note that the constant coefficient of $\hat a_n$ is odd. Set 
\[\hat \alpha_n := \resultant(\Phi_n,\hat a_n) = 2^{-\nu_n\varphi(n)} \alpha_n \quad \text{so that}\quad v_2(\alpha_n) = v_2(\hat \alpha_n) + \nu_n\varphi(n).\] 

%\begin{conjecture}
%If $n\in \N$ is an odd integer or if $n+2$ is a power of $2$,  then $v_2(\hat \alpha_n)=0$. 
%\end{conjecture}
%
%\begin{conjecture}
%If $n\in \N$ is such that $n \equiv 2\pmod{4}$ and $n+2$ is not a power of $2$, then $v_2(\hat \alpha_n)=\varphi(n)$. 
%\end{conjecture}
%
%Those conjectures are particular cases of the following one. 

\begin{conjecture}
If $n = 2^rk\in \N$ with $r\in \N$ and $k$ odd, then $\varphi(k)$ divides $v_2(\hat \alpha_n)$. Additionally,
\begin{enumerate}
\item if $r=0$ then $v_2(\hat \alpha_n)=0$; 
\item if $r=1$ then $\ds v_2(\hat \alpha_n)=\begin{cases} 0&\text{if }k=2^s-1\text{ with }s\in \N^\ast\\ \varphi(k) &\text{otherwise};\end{cases}$ 
\item \label{conj:2r(2s-1)} if $r\in \N$ and $k=2^s-1$ with $s\in \N^\ast$, then $v_2(\hat \alpha_n) = \bigl((r-2)2^{r-1}+1\bigr)  \varphi(k)$.
\end{enumerate}
\end{conjecture}

We prove Case  \eqref{conj:2r(2s-1)} of this conjecture  for $r=0$ in \S\ref{sec:alphanr} and for $r=1$ in \S\ref{sec:2(2s-1)}. In both cases, the conjecture asserts that  $v_2(\hat \alpha_n) =0$, i.e., that $\hat\alpha_n$ is odd.

%
%
%\begin{conjecture}
%If $r\in \N^\ast$, then  $v_2(\alpha_{2^r})=(r-1)2^{r-1}+1$. 
%\end{conjecture}
%
%\begin{conjecture}
%If $r\in \N^\ast$, then  $v_2(\alpha_{2^r\cdot 3})=r2^r+2$. 
%\end{conjecture}

\subsection{The polynomial $\hat a_n$\label{sec:hatan}}

We will now prove that $\hat a_n(\lambda) \in \Z[\lambda]$. Before, we need some preliminary results. 

\begin{prop}
For $n \in \N$, the integer $a_n(0)=g_n(0)$ is the $n$-th Catalan number
\[
a_n(0)=g_n(0)=\frac{(2 n)!}{n!(n+1)!}.
\]
\end{prop}

We give a first proof here. We give a second proof in \S\ref{sec:an0}. 

\begin{proof}[First proof]
 Since $\gauss{n}{j}(0)=1$, 
\[
a_{0}(0)=1 \quad \text {and} \quad a_{n+1}(0)=\sum_{j=0}^n a_j(0) a_{n-j}(0),
\]
which is the recursion formula of Catalan numbers. 
\end{proof}

It follows that  $\nu_n := v_2\bigl(a_n(0)\bigr)$ is the $2$-adic valuation of the $n$-th Catalan number. A result of Alter and Kubota \cite{ak} asserts that this $2$-adic valuation is 
\[\nu_n = \sigma_2(n+1) -1\]
where  $\sigma_2:\N\to \N$ is the function defined by 
\[\sigma_2\left( \sum_{j=0}^m \kappa_j 2^j\right)  := \sum_{j=0}^m \kappa_j\quad \text{when}\quad \kappa_j \in \{0,1\}\text{ for every }j\in \ob 0;m\cb.\] 
In other words, $\sigma_2(k)$ is the sum of the digits in the binary expansion of $k$. This is also the number of occurrences of the digit $1$ in the binary expansion of $k$. 

Multiplying an integer $k$ by $2$ preserves the sum of digits in the binary expansion of $k$ (one only adds a digit $0$), and thus 
\[\forall k\in \N, \quad \sigma_2(2k) = \sigma_2(k).\]
The following result is due to Kummer  \cite{k}. 

\begin{lemma}
 For every $k,m\in \N$,
 \[\sigma_2(k) + \sigma_2(m) =\sigma_2(k+m) + \sum \eps_j,\] 
 where $\eps_j$ denote the carries occurring in the base-2 addition of $k$ and $m$.
 \end{lemma}

\begin{proof}
Let $k\in \N$ and $m\in \N$ be two integers and set $r := k+m\in \N$. Consider the binary expansions  
\[k = \sum_{j\geq 0} \kappa_j 2^j, \quad m = \sum_{j\geq 0} \mu_j 2^j\ \quad \text{and}\quad r = \sum_{j\geq 0} \rho_j 2^j\quad \text{with}\quad \kappa_j,\mu_j,\rho_j\in \{0,1\}\text{ for every }j.\]
The sums are in fact finite sums. Let $\eps_j$ denote the carries occurring in the base-2 addition of $k$ and $m$: 
\[\kappa_0 + \mu_0 = \rho_0 + 2 \eps_0\quad\text{and}\quad \forall j\geq 0,\quad \kappa_{j+1} +  \mu_{j+1} + \eps_j = \rho_{j+1} + 2 \eps_{j+1}.\]
Therefore, 
\[\sigma_2(k) + \sigma_2(m) - \sigma_2(r) =  \sum_{j\geq 0} (\kappa_j + \mu_j - \rho_j )= 2 \eps_0 + \sum_{j\geq 0} (2\eps_{j+1} - \eps_j) = \sum_{j\geq 0} \eps_j.\qedhere\]
\end{proof} 

In particular,  
\[\forall k\in \N,\quad \forall m\in \N, \quad \sigma_2(k) + \sigma_2(m) \geq \sigma_2(k+m).\]
We may now prove that $\hat a_n(\lambda)\in \Z[\lambda]$ for every $n\in \N$. 

\begin{prop}
For every $n\in \N$,  $v_2(a_n) = \nu_n := v_2\bigl(a_n(0)\bigr)$. 
\end{prop}

\begin{proof}
We proceed by induction. We have $\hat a_{0}(\lambda) = 1$ so that $v_2(a_0) = 0 = \nu_0$. Now, assume that $n \geq 0$ and that $v_2(a_k) = \nu_k$ for every $k \in \ob0;n\cb$. The recursion formula~\eqref{eq:recursion} for the polynomials $a_k$ yields
\[a_{n+1} = \begin{cases} 2 \sum\limits_{j = 0}^{m} \gauss{n}{j} a_{j} a_{n -j}& \text{if } n=2m+1 \text{ is odd}\\ 
2 \sum\limits_{j = 0}^{m -1} \gauss{n}{j} a_j a_{n -j} +\gauss{n}m a_m^2 & \text{if } n=2m \text{ is even}. \end{cases}\] 
%The Gaussian polynomials take the value $1$ at $0$, and so their $2$-adic valuations are all equal to $0$. 
For every $j \in \ob0;m\cb$, the induction hypothesis yields
\[v_2\left( 2 \gauss{n }{j} a_j a_{n -j} \right) \geq  1 + \nu_j +\nu_{n -j} = \sigma_2(j +1) +\sigma_2(n -j+1) -1 \geq \sigma_2(n+2) -1 = \nu_{n+1}.\] 
In addition, if $n=2m$ is even, then 
\[v_2\left( \gauss{n}{m} a_m^2 \right) \geq  2 \nu_m= 2 (\sigma_2(m+1) -1) = 2 (\sigma_2(2m+2) -1) = 2(\sigma_2(n+2)-1) = 2\nu_{n+1}\geq \nu_{n+1}.\] 
It follows that $v_2(a_{n+1})\geq \nu_{n+1}$. Since $v_2\bigl(a_{n+1}(0)\bigr) = \nu_{n+1}$, we necessarily have $v_2(a_{n+1}) = \nu_{n+1}$. 
\end{proof}

We will now prove a few results regarding the resultants $\alpha_n$. 

\subsection{When $n+1$ is a power of $2$\label{sec:alphanr}}

\begin{prop}\label{prop:alphanr}
Assume $n +1 = 2^s$ for some integer $s\geq 1$. Then, $\alpha_n$ is odd. 
%If $s\in \N^\ast$, then $\alpha_{2^s-1}$ is odd. 
\end{prop}

\begin{proof}
Consider the sequence $(n_s)_{s\geq 0}$ defined by $n_s := 2^s-1$. 
For every $s\geq 0$, $n_{s+1}-1 = 2 n_s$, so that 
\[a_{n_{s+1}} = a_{2n_s+1} = \sum_{k=0}^{2 n_s} \gauss{2n_s}{k} a_k a_{2n_s-k} =  \gauss{2n_s}{n_s} a_{n_s}^2 + 2 \sum_{k=0}^{n_s-1} \gauss{2n_s}{k} a_k a_{2n_s-k}.\]
It follows that 
\[\forall s\geq 0,\quad a_{n_{s+1}} \equiv \gauss{2n_s}{n_s} a_{n_s}^2 \pmod{2}.\]
Since $a_0 = 1$, we deduce that 
\begin{equation}\label{eq:anr}
\forall s\geq 0,\quad a_{n_s} \equiv \prod_{k=0}^{s-1}  \gauss{2n_k}{n_k}^{2^{s-k-1}} \pmod{2}.
\end{equation}
For every $k\in \ob 0;s-1\cb$, the Gaussian polynomial $ \gauss{2n_k}{n_k}$ is a (possibly empty) product of cyclotomic polynomials $\Phi_m$ with $m\in \ob 1;2n_{s-1}\cb$. In that case,  $m\leq 2n_{s-1} = n_s-1$ and the resultant of $\Phi_m$ and $\Phi_{n_s}$ is equal to $1$ or the power of a prime divisor of $n_s$. Since $n_s$ is odd, this resultant is odd. It follows that $\alpha_{n_s}$ is a product of odd numbers, thus an odd number. 
\end{proof}

\subsection{When $n+1$ is the sum of two powers of $2$}

\begin{prop}
Assume $n +1 = 2^r+2^s$ for some integers $r>s\geq 1$. Then, $\hat \alpha_n$ is odd. 
%Assume $r>s\geq 1$ are integers and $n = 2^r+2^s-1$. Then, $v_2(\alpha_n) = \varphi(n)$. 
\end{prop}

\begin{proof}
%Assume that $n = 2^r +2^s -1$, with $r > s \geq 1$. In that case, $\nu_n = \sigma_2(n+1) -1 = 1$, so that
%\[v_2(a_n) = 1,\quad a_n = 2\hat a_n\quad\text{and}\quad v_2(\alpha_n) = v_2(\hat \alpha_n) + \varphi(n).\] 
%It suffices to prove that $\hat \alpha_n$ is an odd integer. 
As in the proof of Proposition~\ref{prop:alphanr}, consider the sequence $(n_j)_{j\geq 0}$ defined by $n_j := 2^j-1$. Set $m:= n_{r-1} + n_{s-1}$, so that $n-1=2(m+1)$. 
By the recursion formula~\eqref{eq:recursion}, 
\[a_n =\gauss{n-1}{m+1} a_{m+1}^2 + 2\sum_{j=0}^{m} \gauss{n-1}{j} a_j a_{n-j-1}.\]
The Gaussian polynomials have integer coefficients and are not divisible by $2$ since their value at $0$ is $1$. Consequently, 
\[v_2\left(\gauss{n-1}{m+1} a_{m+1}^2\right) = 2 v_2(a_{m+1}) = 2\bigl(\sigma_2(2^{r-1} + 2^{s-1})-1\bigr)=2.\]
In addition, for $j\in \ob0;m\cb$,
\begin{eqnarray*}
v_2\left(2\gauss{n-1}{j} a_j a_{n-j-1}\right) &=& 1 + v_2(a_j) + v_2(a_{n-j-1}) \\
&=& \sigma_2(j+1) + \sigma_2(n-j) -1 \geq \sigma_2(n+1) - 1 = 1. 
\end{eqnarray*}
Equality occurs if and only if there are no carries in the base-2 addition of $j+1$ and $n-j$, i.e., if and only if $j+1= 2^s = n_s+1$ and $n-j = 2^r = n_r+1$. 
It follows that 
\[\hat a_n \equiv  \gauss{n-1}{n_s} \hat a_{n_s} \hat a_{n_r} \pmod{2}.\]
It now follows from Equation~\eqref{eq:anr} that the reduction modulo $2$ of $\hat a_n$ is a product of cyclotomic polynomials $\Phi_j$ with $j\in \ob1;n-1\cb$. As $n$ is odd, these cyclotomic polynomials are all coprime to $\Phi_n$ modulo $2$. Hence, $\hat\alpha_n$ is an odd integer.
%, so that 
%\[v_2(\alpha_n) = v_2(\hat\alpha_n) + \varphi(n) = \varphi(n).\qedhere\]
\end{proof}

\subsection{A lower bound on $v_p(\alpha_n)$\label{sec:valuation}}

In this section, we prove that $n$ divides $\alpha_n$ for all $n\in \N^\ast$. It took us quite a while to realize that this is a general phenomenon, rather than a property specific to the quadratic polynomial $\lambda z(1-z)$. 

\begin{prop}\label{prop:valuation}
For every $n\in \N^\ast$ and every prime number $p$ dividing $n$, 
\[v_p(\alpha_n) \geq \frac{\varphi(n)}{p-1}, \]
where $\varphi$ is the Euler totient function. 
\end{prop}

\begin{coro}
For every $n\in \N^\ast$, $n$ divides $\alpha_n$. 
\end{coro}

\begin{proof}
Fix $n\in \N^\ast$ and let $p$ be a prime number dividing $n$. Let $n = kp^r\in \N^\ast$ where $k\in \N^\ast$ is not divisible by  $p$ and $r\geq 1$.  By Proposition~\ref{prop:valuation}, 
\[v_p(\alpha_n) \geq \frac{\varphi(n)}{p-1} = \varphi(k) p^{r-1}\geq 1\cdot 2^{r-1} = (1+1)^{r-1}\geq 1+(r-1) = r = v_p(n).\]
Thus, $p^r$ divides $\alpha_n$. Since this is true for all prime divisors of $n$, this completes the proof. 
\end{proof}

Our proof of Proposition~\ref{prop:valuation} relies on the following result. Set 
\[\Psi_n = \prod_{\substack{d \mid n \\ d\neq n}} \Phi_d.\]

\begin{lemma}\label{lem:valres}
For all $n\in \N^\ast$ and every prime number $p$ dividing $n$, 
\[v_p\bigl(\resultant(\Phi_n,\Psi_n)\bigr)= \frac{\varphi(n)}{p-1}.\]
\end{lemma}

\begin{proof}
Fix $n\in \N^\ast$ and let $p\geq 2$ be a prime number dividing $n$. Write $n=kp^r$ where $k\in \N^\ast$ is not divisible by $p$ and $r\in \N^\ast$. Then, 
\[\frac{\varphi(n)}{p-1} = \varphi(k)p^{r-1}.\]
Now, 
\[\resultant(\Phi_n,\Psi_n) =  \prod_{\substack{d \mid n \\ d\neq n}} \resultant(\Phi_n,\Phi_d).\]
According to  Apostol \cite{apostol}:  for all $m,n\in\N^\ast$ with $m<n$, we have
\[\resultant(\Phi_m,\Phi_n) = 
\begin{cases} 
p^{\varphi(m)}&\text{if }n = mp^s\text { with } p\geq 2 \text{ prime and }s\in\N^\ast, \\ 
1& \text{otherwise.}
\end{cases}\]
The divisors $d\in \N^\ast$ of $n$ for which $n = dp^s$ with $s\in \N^\ast$ are precisely the integers $d=kp^{r-s}$ with $s\in \ob1,r\cb$. 
Consequently, 
\begin{eqnarray*}
v_p\bigl(\resultant(\Phi_n,\Psi_n)\bigr)&=&\varphi(k)\left(1 + (p-1) + (p-1)p +\cdots + (p-1)p^{r-2}\right) \\
&=&  \varphi(k)\left(1+ (p-1)\frac{p^{r-1}-1}{p-1}\right) = \varphi(k)p^{r-1} = \frac{\varphi(n)}{p-1}.\qedhere
\end{eqnarray*}
\end{proof}

Our proof also relies on a result  of Morton and Patel \cite{mp}. We need a slightly stronger version (see \cite[Theorem 13]{bfls}). 

\begin{theo}\label{theo:mp}
Let $P\in R[z]$ be a polynomial with coefficients in a unique factorization domain $R$. 
Then, there exists a sequence of polynomials $\bigl(\Phi_n^P\in R[z]\bigr)_{n\in \N^\ast}$ such that for all $n\in \N^\ast$, 
\[P^{\circ n}(z) - z = \prod_{d\mid n} \Phi_d^P.\]
\end{theo}

%\begin{coro}\label{coro:mp}
%Let $P\in R[z]$ be a monic polynomial with coefficients in some integral domain $R$. Then, there exists a sequence of monic polynomials $\bigl(\Phi_n^P\in R[z]\bigr)_{n\in \N^\ast}$ such that for all $n\in \N^\ast$, 
%\[P^{\circ n}(z) - z = \prod_{d\mid n} \Phi_n^P.\]
%\end{coro}
%
%\begin{proof}
%Let $K$ be the field of fractions of $R$. Let $(\Phi_n^P)_{n\in \N^\ast}$ be the sequence provided by Theorem~\ref{theo:mp}.  Since $P$ is monic, and since the Euclidean division algorithm works in any ring, we get the result by induction on $n$. 
%\end{proof}

\begin{proof}[Proof of Proposition \ref{prop:valuation}]
On the one hand, for $n\in \N^\ast$, 
\[F_\lambda^{\circ n}(z) = \lambda^n z C_n(\lambda,z)\quad \text{with}\quad C_n(\lambda,z) := \sum_{k\geq 0} c_{n,k}(\lambda) z^k.\]
We have
\[\alpha_n = \resultant(\Phi_n, -c_{n,n}).\]
On the other hand, set $R:= \Z[\lambda]$ and regard the family $F_\lambda$ as a polynomial $P\in R[z]$: 
\[P(z) := \lambda z - \lambda z^2.\]
Let $\bigl(\Phi_n^P\in R[z]\bigr)_{n\in \N^\ast}$  be the sequence of polynomials provided by Theorem~\ref{theo:mp}. 
Then, 
\[\Phi_1^P(z) = P(z) - z = (\lambda-1) z - \lambda z^2 = z D_1(\lambda,z) \quad \text{where}\quad D_1(z) := \lambda-1 - \lambda z.\]
For $n\geq 2$, set 
\[D_n(\lambda,z) := \Phi_n^P(z).\]
Then, for all $n\in \N^\ast$, 
\begin{equation}\label{eq:cndn}
\lambda^n C_n(\lambda,z)  -1=  \prod_{d\mid n} D_d(\lambda,z).
\end{equation}
For $n\in \N^\ast$, define $E_n\in \Z[\lambda,z]$ by 
\[E_n := \prod_{\substack{d\mid n\\d\neq n}} D_d, \quad \text{so that}\quad \lambda^n C_n -1= D_n\cdot  E_n.\]
Write
\[D_n(\lambda,z) = \sum_{k\geq 0} d_{n,k}(\lambda) z^k\quad \text{and}\quad E_n(\lambda,z) = \sum_{k\geq 0} e_{n,k}(\lambda) z^k\quad \text{with}\quad d_{n,k},e_{n,k}\in \Z[\lambda].\]
Then, $c_{n,0} = 1$ and for all $k\in \N^\ast$, 
\[\lambda^n c_{n,k} = \sum_{\substack{i,j\geq 0\\i+j= k}} d_{n,i} e_{n,j}.\]

\begin{lemma}
For all $n\in \N^\ast$, 
\[d_{n,0} = \Phi_n.\]
%\quad \text{and}\quad e_{n,0} = \Psi_n.\]
\end{lemma}

\begin{proof}
By definition, $d_{1,0}  = D_1(\lambda,0) = \lambda-1 = \Phi_1(\lambda)$. And since $F_\lambda^{\circ n}(z) - z = (\lambda^n-1) z + \O(z^2)$, for all $n\in \N^\ast$, 
\[\prod_{d\mid n} d_{n,0} = \prod_{d\mid n} D_n(\lambda,0) = \lambda^n C_n(\lambda,0) -1= \lambda^n-1 = \prod_{d\mid n} \Phi_n(\lambda).\]
The result for $d_{n,0}$ follows by induction on $n$.%, and the result for $e_{n,0}=E_n(\lambda,0)$ then follows from the definition of $E_n$. 
\end{proof}

%\begin{lemma}\label{lem:phindividesc}
%For all $k\in \ob0,n-1\cb$, the monic polynomial $\Phi_n$ divides $c_{n,k}$ in $\Z[\lambda]$. 
%\end{lemma}
%
%\begin{proof}
%Let $\omega\in \Omega_n$ be a primitive $n$-th root of unity. Then, 
%\[F_\omega^{\circ n}(z) = z + \O(z^{n+1}) = z \sum_{k\geq 0} c_{n,k}(\omega) z^k.\]
%Assume $k\in  \ob0,n-1\cb$. It follows that $c_{n,k}(\omega) = 0$. Hence, $\Phi_n$ divides $c_{n,k}$ in $\Qbar[z]$. Since $\Phi_n$ is a monic polynomial in $\Z[\lambda]$ and since $c_{n,k}\in \Z[\lambda]$, $\Phi_n$ divides $c_{n,k}$ in $\Z[\lambda]$. 
%\end{proof}

\begin{lemma}\label{lem:phindividesd}
For all $k\in \ob0,n-1\cb$, the polynomial $\Phi_n$ divides $d_{n,k}$ in $\Z[\lambda]$. 
\end{lemma}

\begin{proof}
Let $\omega \in \Omega_n$ be a primitive $n$-th root of unity. Then, 
\[z D_n(\omega, z) E_n(\omega, z) = F_\omega^{\circ n}(z) -z = \O(z^{n+1}).\] In addition, by the previous lemma,
\[ E_n(\omega, 0) = \prod_{\substack{d\mid n\\d\neq n}} d_{n,0}(\omega) = \Psi_n(\omega) \neq 0.\] Therefore, $D_n(\omega, z) = \O(z^n)$. Assume $k\in  \ob0,n-1\cb$. Thus, we have $d_{n,k}(\omega) = 0$ for all $\omega\in \Omega_n$. Since $\Phi_n$ is a monic polynomial in $\Z[\lambda]$ and $d_{n,k} \in \Z[\lambda]$, it follows that $\Phi_n$ divides $d_{n,k}$ in $\Z[\lambda]$.
\end{proof}

%\begin{proof}
%We prove the result by strong induction on $k$. For $k=0$: $d_{n,0} = \Phi_n$ by the previous lemma, so the property holds trivially. Assume that $k\in \ob 1, n-1\cb$ and $\Phi_n$ divides $d_{n,i}$ for all $i\in \ob0,k-1\cb$. Then, 
%\[\lambda^n c_{n,k} = \sum_{\substack{i,j\geq0\\i+j = k}} d_{n,i} e_{n,j} = d_{n,k} e_{n,0} + b_{n,k},\quad \text{where}\quad b_{n,k} :=\sum_{i=0}^{k-1} d_{n,i} e_{n,k-i}.\]
%By the induction hypothesis, $\Phi_n$ divides $b_{n,k}$ in $\Z[\lambda]$, and by Lemma~\ref{lem:phindividesc}, it divides $c_{n,k}$ in $\Z[\lambda]$. Therefore, $\Phi_n$ divides $d_{n,k} \Psi_n = c_{n,k} - b_{n,k}$. Since $\Phi_n$ and $\Psi_n$ are coprime, and since $\Phi_n$ is a monic polynomial, it divides $d_{n,k}$ in $\Z[\lambda]$. 
%\end{proof}

It follows that 
\[\lambda^n c_{n,n} =  \sum_{\substack{i,j\geq0\\i+j = n}} d_{n,i} e_{n,j} \equiv d_{n,n} \Psi_n \pmod{\Phi_n}.\]
Thus, 
\[\underset{=1}{\underbrace{\resultant(\Phi_n,\lambda^n)}}\ \alpha_n = \resultant(\Phi_n,\Psi_n)\cdot \underset{\in \Z}{\underbrace{\resultant(\Phi_n,-d_{n,n})}}.\]
This completes the proof since $v_p\bigl( \resultant(\Phi_n,\Psi_n)\bigr) = \ds \frac{\varphi(n)}{p-1}$ by Proposition~\ref{lem:valres}. 
\end{proof}

\subsection{When $n$ is a prime number\label{sec:prime}}

\begin{prop}\label{prop:nprime}
If $n\geq 2$ is a prime number, then $\alpha_n \equiv n \pmod{n^2}$. 
\end{prop}

The remainder of this section is devoted to the proof of this proposition. 
We first  compute $a_k(1)$ and $a_k'(1)$ for arbitrary integers $k\geq 1$. 

\begin{lemma}
For every integer $k\geq 1$,
\[a_k(1) = k!.\]
\end{lemma}

\begin{proof}
The proof goes by induction. The property is true for $k=0$. And if the property holds for $j\in \ob0;k\cb$, then 
\[a_{k+1}(1) = \sum_{j=0}^k \gauss{k}{j}(1) a_j(1) a_{k-j}(1) =  \sum_{j=0}^k\frac{k!}{j!(k-j)!} j! (k-j)! = (k+1)k! = (k+1)!.\qedhere\]
\end{proof}

The determination of $a_k'(1)$ is more tricky. 

\begin{lemma}\label{lem:akprime1}
For every integer $k\geq 1$,
\[a_k'(1) = k!\left(\frac{k(k-1)}{4} - \sum_{j=0}^{k-2} \frac{j+1}{k-j}\right).\]
\end{lemma}

\begin{proof}
First, consider the change of coordinates  $z = (1-\lambda) w$. 
In those coordinates, the map $F_\lambda$ is conjugate to 
\[{\widetilde F}_\lambda(w) = w + (\lambda-1)w (1 +\lambda w)  = w + (\lambda-1) \phi(w) + \O\bigl((\lambda-1)^2\bigr) \quad \text{with}\quad \phi(w) = w(1+w).\]
For $\lambda\neq 1$, the linearizing map ${\widetilde G}_\lambda $ of  ${\widetilde F}_\lambda$ which is tangent to the identity at the origin is the map 
\[{\widetilde G}_\lambda(x) = \frac{G_\lambda\bigl((1-\lambda) x\bigr)}{1-\lambda} = x  \sum_{k\geq 0} \gamma_k(\lambda)x^k\quad \text{with}\quad \gamma_k(\lambda) := (1-\lambda)^k g_k(\lambda) = a_k(\lambda) \frac{(1-\lambda)^k}{\pi_k(\lambda)} .\]
Note that 
\[\lim_{\lambda \to 1} \frac{\pi_k(\lambda)}{(1-\lambda)^k} = k!,\quad \text{so that}\quad\forall k\geq 0, \quad \gamma_k(1) = 1.\]
In particular, 
\[{\widetilde G}_1(x) := \lim_{\lambda\to 1} {\widetilde G}_\lambda(x) = x \sum_{k\geq 0} x^k =  \frac{x}{1-x}.\]

\begin{remark}
The map ${\widetilde G}_1$ satisfies the differential equation $\phi\circ {\widetilde G}_1(x)   = x {\widetilde G}_1'(x)$. Hence, it linearizes the vector field $\phi(w) \partial_w$, i.e., if we make the change of variables $w = {\widetilde G}_1(x)$, then $\phi(w) \partial_w = x\partial_x$.
\end{remark}

In the coordinate $x$, the map $F_\lambda$ is conjugate to a map $K_\lambda$: if $z = (1-\lambda)w = (1-\lambda) {\widetilde G}_1(x)$, then $F_\lambda(z) = (1-\lambda) {\widetilde F}_\lambda(w) = (1-\lambda) {\widetilde G}_1\circ K_\lambda(x).$
An elementary computation shows that 
\[K_\lambda (x) = \lambda\left(x + (\lambda-1)^2\kappa(x) + \O((\lambda-1)^3)\right)\quad \text{with}\quad \kappa(x)  = -\frac{x^3}{1-x}.\]
Let $H_\lambda$ be the linearizing map of $K_\lambda$  given by ${\widetilde G}_\lambda = {\widetilde G}_1\circ H_\lambda$: 
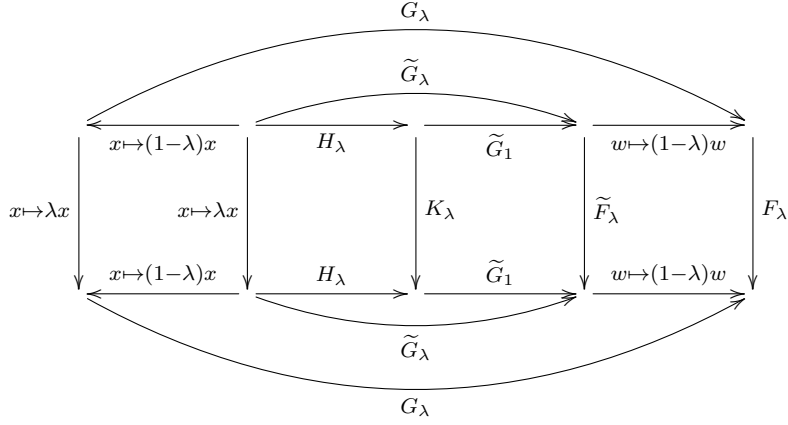
\[\xymatrix@C=2cm@R=2cm{
\ar@/^3pc/[rrrr]^{G_\lambda}  \dto_{x\mapsto \lambda x} & \lto^{x\mapsto (1-\lambda)x} \dto_{x\mapsto \lambda x} \rto_{H_\lambda}  \ar@/^1pc/[rr]^{{\widetilde G}_\lambda}  & \dto^{K_\lambda} \rto_{{\widetilde G}_1} & \dto^{{\widetilde F}_\lambda}  \rto_{w\mapsto (1-\lambda) w} &  \dto^{F_\lambda}  \\ 
\ar@/_3pc/[rrrr]_{G_\lambda}  & \lto_{x\mapsto (1-\lambda)x}  \rto^{H_\lambda}  \ar@/_1pc/[rr]_{{\widetilde G}_\lambda}   & \rto^{{\widetilde G}_1} &  \rto^{w\mapsto (1-\lambda) w}&
}\]
Setting $\eps:=\lambda-1$, we may write 
\[H_\lambda(x) = x + \eps \eta(x) + \eps^2 \theta(x) + \O\bigl(\eps^3\bigr)\]
with $\eta(x)\in x^2\Q[[x]]$ and $\theta(x)\in x^2\Q[[x]]$.
Note that 
\begin{eqnarray*}
H_\lambda(\lambda x) &=& x +\eps  x + \eps \eta(x+\eps x) + \eps^2 \theta(x) + \O(\eps^3)\\
&=& x + \eps \bigl(x + \eta(x)\bigr) + \eps^2 \bigl(x\eta'(x) + \theta(x)\bigr)+ \O(\eps^3)
\end{eqnarray*}
and 
\begin{eqnarray*}
K_\lambda\circ H_\lambda(x) &=& (1+\eps) \bigl( x+ \eps \eta(x) + \eps^2 \theta(x)+\eps^2\kappa(x)\bigr)+\O(\eps^3)\\
&=& x + \eps \bigl(x + \eta(x)\bigr) + \eps^2 \bigl(\eta(x) + \theta(x)+\kappa(x)\bigr)+ \O(\eps^3).
\end{eqnarray*}
It follows that 
\[x\eta'(x) - \eta(x) = \kappa(x).\]

\begin{remark}
This equation expresses the fact that the Lie bracket of the vector fields $x\partial x$ and $\eta(x)\partial_x$ is the vector field $\kappa(x)\partial_x$. 
\end{remark}

Since 
\[\kappa(x) = -\frac{x^3}{1-x} = x  \sum_{k\geq 2} -x^k,\]
we deduce that 
\[\eta(x) = x  \sum_{k\geq 2} \eta_k x^k\quad \text{with}\quad (k+1) \eta_k - \eta_k = -1, \quad \text{i.e.,}\quad \eta_k = -\frac{1}{k}.\]
Since ${\widetilde G}_\lambda  = {\widetilde G}_1\circ H_\lambda$, we deduce that 
\[ x  \sum_{k\geq 1} \gamma_k'(1) x^k = {\widetilde G}_1'(x)\eta(x).\]
Observe that 
\[{\widetilde G}'_1(x) = \sum_{j\geq 0} (j+1) x^j\quad\text{so that}\quad 
\gamma_k'(1) = - \sum_{j=0}^{k-2} \frac{j+1}{k-j}.\]
Finally, since $a_k(\lambda) = \gamma_k(\lambda)  \ds \prod_{j=1}^k (1+\lambda + \ldots +\lambda^{j-1})$, 
\[\frac{\gamma_k'(\lambda)}{\gamma_k(\lambda)} = \frac{a_k'(\lambda)}{a_k(\lambda)} - \sum_{j=1}^k \frac{1+2\lambda + \cdots +  (j-1) \lambda^{j-2}}{1+\lambda + \cdots + \lambda^{j-1}}.
 \]
It follows that 
\[\gamma_k'(1) = \frac{\gamma_k'(1)}{\gamma_k(1)} = \frac{a_k'(1)}{a_k(1)} - \sum_{j=1}^k \frac{j-1}{2} = \frac{a_k'(1)}{k!} - \frac{k(k-1)}{4},\]
whence
\[a_k'(1) = k!\left(\frac{k(k-1)}{4} - \sum_{j=0}^{k-2} \frac{j+1}{k-j}\right).\qedhere\]
\end{proof}

We may now return to the proof of Proposition~\ref{prop:nprime}. We first observe that $\alpha_2 = 2 \equiv 2\pmod{2^2}$. We may therefore assume that the prime number $n$ is at least $3$. 

\begin{lemma}\label{lem:simpleroot}
If $n\geq 3$ is a prime number, then $a_n(1) \equiv 0\pmod{n}$ and $a'_n(1) \equiv 1\pmod{n}$. 
\end{lemma}

\begin{proof}
Assume that $n\geq 3$ is a prime number. From $a_n(1) = n!$, we deduce that $a_n(1) \equiv 0\pmod{n}$. 
By Lemma~\ref{lem:akprime1} above, 
\[a_n'(1) = n!\left(\frac{n(n-1)}{4} - \sum_{j=0}^{n-2} \frac{j+1}{n-j}\right).\]
Note that $n$ divides $\ds n! \frac{n(n-1)}{4}$ and the terms $\ds n! \frac{j+1}{n-j}$ when $j\in \ob1;n-2\cb$. Using the Wilson Theorem, we deduce that  
\[a_n'(1) \equiv -n! \frac{0+1}{n-0} \equiv -(n-1)!  \equiv  1 \pmod{n}.\qedhere\]
\end{proof}

The proof of Proposition~\ref{prop:nprime} is now a consequence of the following lemma. This result is probably well-known but we did not find a reference in the literature. 

\begin{lemma}
Assume $p\geq 3$ is a prime number and  $a(\lambda)\in \Z[\lambda]$ satisfies 
\[a(1) \equiv  0\pmod{p}\quad \text{and}\quad a'(1) \equiv 1\pmod{p}.\] 
Then, $\resultant(\Phi_p,a) \equiv p\pmod{p^2}$.
\end{lemma}

\begin{proof}
Let $p\geq 3$ be a prime number.  For every $\omega\in \Omega_p$,
\[p = \Phi_p(1) = \prod_{k=1}^{p-1}(1-\omega^k) =  \prod_{k=1}^{p-1} \left((1-\omega) \sum_{j=0}^{k-1}\omega^j\right)  = (1-\omega)^{p-1}b(\omega),\]
where 
\[b(\lambda) := \prod_{k=1}^{p-1} \left(\sum_{j=0}^{k-1} \lambda^j\right)\in \Z[\lambda].\]
Since $a(1) \equiv  0\pmod{p}$ and $a'(1) \equiv 1\pmod{p}$, there exist $r,s\in \Z$ such that for every $\omega\in \Omega_p$, 
\[a(1) = r p =  (1-\omega)^2 b_0(\omega) \quad \text{and}\quad a'(1) = 1+sp = 1+ (1-\omega) b_1(\omega)\]
where
\[b_0(\lambda) := r \cdot (1-\lambda)^{p-3} b(\lambda)\in\Z[\lambda] \quad \text{and}\quad b_1(\lambda) := s\cdot (1-\lambda)^{p-2} b(\lambda)\in \Z[\lambda].\]
In addition, we may write
\[a(\lambda) = a(1) +  a'(1) (\lambda-1)  + (\lambda-1)^2 b_2(\lambda) \quad \text{with}\quad b_2 \in \Z[\lambda].\]
Thus, for every $\omega\in \Omega_p$, 
\[a(\omega) = (\omega-1)\bigl(1+ (\omega-1) c(\omega)\bigr)\quad \text{with}\quad  c = b_0-b_1+b_2\in \Z[\lambda].
\]
Therefore, 
\[\resultant(\Phi_p,a) = \resultant(\Phi_p,\lambda-1)\cdot \resultant\bigl(\Phi_p,1 + (\lambda-1)c(\lambda)\bigr) \equiv p\pmod{p^2}\]
since
\[\resultant(\Phi_p,\lambda-1)  = (-1)^{p-1}\Phi_p(1) = p\]
and
\[\resultant\bigl(\Phi_p,1 + (\lambda-1)c(\lambda)\bigr) \equiv \resultant\left((\lambda-1)^{p-1},1+(\lambda-1)c(\lambda)\right) \equiv 1\pmod{p}\]
because $\Phi_p(\lambda) \equiv (\lambda-1)^{p-1} \pmod{p}$. This completes the proof of the lemma. 
\end{proof}

\subsection{Specific values of the polynomials $a_n$}

We have seen that $a_n(1) = n!$ for every $n\in \N$ and following Yoccoz \cite{yoccoz}, it is possible to give an explicit formula of $a_n(-1)$ for every $n\in \N$. 
It is also possible to determine an explicit value of $a_n(\lambda)$ for some other integers $\lambda\in \Z$, namely, for $\lambda=0, \lambda= \pm 2$ and $\lambda=4$. In those three cases, the linearizing map $G_\lambda$ is explicit.

\subsubsection{Values at roots of unity\label{sec:valueroots}}

Given an integer $n\geq 1$, we wish to study the values of $a_n$ at roots of unity which are not necessarily primitive $n$-th roots of unity. Here, $\zeta\in \Omega_q$ denotes a primitive $q$-th root of unity for some integer $q\geq 1$, and as previously, $\omega\in \Omega_n$ denotes a primitive $n$-th root of unity. 

The following result is due to Yoccoz \cite{yoccoz}. We include a proof for completeness. 

\begin{prop}\label{prop:yoccoz}
Assume $q\geq 1$ is an integer, and $\zeta\in \Omega_q$. Then,  
\[\forall n=mq+s\in \N\quad \text{with}\quad m\in \N,\quad s\in \ob 0;q-1\cb, \quad a_n(\zeta) =a_s(\zeta)a_q^m(\zeta) \prod_{k=0}^{m-1} (kq+s+1) .\]
\end{prop}

\begin{proof}
Let $q\geq 1$ be an integer and fix $\zeta\in \Omega_q$. We prove the property by induction on $n\in \N$. For $n= 0$, we have $a_0(\zeta)=1$ and the result is trivial. Assume the property holds for every integer less than or equal to some integer $n\geq 0$. Consider the Euclidean division $n = mq+s$ with $m\in \N$ and $s\in \ob 0;q-1\cb$. Using the notation of \S\ref{sec:factorials}, we have
\[ \prod_{k=0}^{m-1} (kq+s+1)  = q^m \left(\frac{s+1}{q}\right)^{(m)}.\]
By Lemma~\ref{lem:gausseuclid}, if $k =jq+r\in \ob0;n\cb$ with $r\in \ob0;q-1\cb$, then 
\[\gauss{n}{k}(\zeta) =\begin{cases} \ds \binom{m}{j} \cdot \gauss{s}{r}(\zeta)& \text{if }r\in \ob 0;s\cb\\ 0 & \text{if } r\in \ob s+1;q-1\cb.\end{cases}\]
Consequently,
\begin{eqnarray*}
a_{n+1}(\zeta) &=& \sum_{j=0}^{m}  \sum_{r=0}^s \binom{m}{j} \cdot \left(\gauss{s}{r} a_{jq+r} a_{(m-j)q+(s-r)}\right)(\zeta)\\
&=&\sum_{j=0}^{m}  \sum_{r=0}^s  \binom{m}{j}  q^j\left(\frac{r+1}{q}\right)^{(j)} q^{m-j}\left(\frac{s-r+1}{q}\right)^{(m-j)} \left( \gauss{s}{r} a_r   a_q^j\cdot a_{s-r} a_q^{m-j} \right)(\zeta)\\
&=& a_q^m(\zeta) q^m \sum_{r=0}^s \left( \gauss{s}{r} a_r  a_{s-r} \right)(\zeta)\cdot  \left(\sum_{j=0}^m  \binom{m}{j}  \left(\frac{r+1}{q}\right)^{(j)}\left(\frac{s-r+1}{q}\right)^{(m-j)}\right).
\end{eqnarray*}
By Lemma~\ref{lem:factorials}, the right sum does not depend on $r$ and $s$ and is equal to $\left(\frac{s+2}{q}\right)^{(m)}$. 
Thus, 
\begin{eqnarray*}
a_{n+1}(\zeta) &=& a_q^m(\zeta) q^m\left(\frac{s+2}{q}\right)^{(m)} \sum_{r=0}^s \left( \gauss{s}{r} a_r  a_{s-r} \right)(\zeta)\\
& =& a_{s+1}(\zeta)a_q^m(\zeta) \prod_{k=0}^{m-1} (kq+s+2).
\end{eqnarray*}
Observe that $n+1 = mq+s+1$.  Either $s+1\leq  q-1$, in which case the induction is completed. Or $s+1 = q$, in which case $n+1 = (m+1)q$ and 
\[a_{n+1}(\zeta) = a_q(\zeta) a_q^m(\zeta) \prod_{j=1}^m (jq+1) = a_q(\zeta) a_q^m(\zeta) \prod_{j=0}^m (jq+1)\]
as required. 
\end{proof}

In the particular cases $q=1$ and  $q=2$, we have the following corollaries. 

\begin{coro}
For every $m\in \N$,
\[a_m(1)=m!.\] 
\end{coro}
 
 \begin{proof}
Since $a_0 = a_1 = 1$, we deduce that $a_m(1) = a_{m\cdot 1+0}(1) = a_0(1) a_1^m(1) m! = m!$. 
 \end{proof}

\begin{coro}
For every $n \in \N$,
\[
a_n(-1)= \begin{cases}\ds \frac{(2 m)!}{m!} & \text { if } n=2 m \text { is even } \\ \ds 4^m m! & \text { if } n=2 m+1 \text { is odd. }\end{cases}
\]
\end{coro}

\begin{proof}
Note that $a_0 = a_1 = 1$ and $a_2=2$. Therefore,
\[a_{2m}(-1) = 2^m \prod_{k=0}^{m-1} (2k+1)= \frac{(2 m)!}{m!} \quad \text{and}\quad a_{2m+1}(-1) = 2^m \prod_{k=0}^{m-1} (2k+2) = 4^m m! .\qedhere\]
\end{proof}

Yoccoz \cite[Lemma 1, page 76]{yoccoz} also asserts that $a_n(\omega)\neq 0$ for any $\omega\in \Omega_n$. The proof he presents is an arithmetic proof. His proof relies on the following assertion: it follows from Proposition~\ref{prop:yoccoz} that if $a_n$ has a root at $\zeta$, then $a_{nm+r}$ has a root of multiplicity at least $m$ at $\zeta$. However, Proposition~\ref{prop:yoccoz} is only valid at $\zeta$, not for $\lambda \neq \zeta$, and we do not see how to prove Yoccoz's assertion. 

Given integers $k\geq 1$ and $n\geq 0$, define
\[\alpha_{k,n} := \resultant(\Phi_k,a_n)\]
so that $\alpha_n = \alpha_{n,n}$ for every integer $n\geq 1$. Taking the product over all $\zeta\in \Omega_q$ in  Proposition~\ref{prop:yoccoz} yields the following result. 

\begin{prop}
Let $q\geq 1$ be an integer. Then, 
\[\forall n=mq+s\in \N\quad \text{with}\quad m\in \N,\quad s\in \ob 0;q-1\cb, \quad 
\alpha_{q,n} =  \alpha_{q,s} \alpha_q^m \prod_{k=0}^{m-1} (kq+s+1)^{\varphi(q)}.\]
\end{prop}

\subsubsection{When $n = 2(2^s-1)$\label{sec:2(2s-1)}}

\begin{prop}\label{prop:r1}
Suppose that $n=2(2^s-1)$ with $s\ge1$. Then $\hat\alpha_n$ is odd.
\end{prop}
  
\begin{proof}
Set $q:=2^s-1$, so that $n=2q$. 
By Proposition~\ref{prop:yoccoz}, there exists $b\in\Z[\lambda]$ such that
\[a_n = (q+1) a_q^{2} + b\cdot\Phi_q = 2^sa_q^{2} + b\cdot\Phi_q.\]
Now, note that $\nu_n=\sigma_2(2^{s+1}-1)-1=s$ and $\nu_q=\sigma_2(2^s)-1=0$,
and hence $a_n=2^s\hat a_n$, $a_q=\hat a_q$ and 
\[2^s \hat a_n = 2^s \hat a_q^2 + b\cdot \Phi_q.\]
Therefore, since
$\Phi_q\in\Z[\lambda]$ is monic, there exists $c\in\Z[\lambda]$
such that $b=2^sc$, and we have
\[\hat a_n = \hat a_q^2 + c\cdot\Phi_q.\]
As a result, since $\Phi_n = \Phi_{2q} \equiv\Phi_q \pmod2$, 
\[\hat\alpha_n \equiv \resultant\big(\Phi_q,\hat a_q^2\big)\equiv \hat\alpha_q^{2} \equiv 1 \pmod2,\]
thanks to Proposition~\ref{prop:alphanr}. Thus, the proposition is proved.
\end{proof}

\subsubsection{Values for $\lambda=0$\label{sec:an0}}

\begin{prop}
For $k \in \N$, the integer $a_k(0)=g_k(0)$ is the $k$-th Catalan number
\[
a_k(0)=g_k(0)=\frac{(2 k)!}{k!(k+1)!}.
\]
\end{prop}

\begin{proof}[Second proof]
Since $F_\lambda=\lambda F_{1}$, we may rewrite the linearizing equation as
\[
\frac{G_\lambda(\lambda z)}\lambda=F_{1} \circ G_\lambda(z).
\]
Passing to the limit as $\lambda$ tends to 0 yields
\[
F_{1} \circ G_{0}(z)=z \quad \text { so that } \quad G_{0}(z)=\frac{1-\sqrt{1-4 z}}{2}
\]
and $G_{0}(z) / z$ is the generating series of the Catalan numbers.
\end{proof}

\subsubsection{Values for  $\lambda=2$}

The following result has been used in  \S\ref{sec:padic}. 

\begin{prop}
For $k \in \N$,
$$
g_k(2)= \frac{(-2)^k}{(k+1)!} \quad \text { and } \quad a_k(2)=\frac{(-2)^k}{(k+1)!} \cdot \prod_{j=1}^k (1-2^j) .
$$
\end{prop}

\begin{proof}
The map $F_2$ is conjugate to $w \mapsto w^2$ via the change of variable $w=1-2 z$. The linearizing map $G_2$ is
\[
G_2(z)=\frac{1-\mathrm{e}^{-2 z}}{2}=\sum_{k \geq 1}-\frac{(-2 z)^{k}}{2 k!}.\qedhere
\]
\end{proof}

\subsubsection{Values for  $\lambda=4$}

\begin{prop}
For $k \in \N$,
\[
g_k(4)=(-1)^k \frac{2^{2 k+1}}{(2 k+2)!} \quad \text { and } \quad a_k(4)=(-1)^k \frac{2^{2 k+1}}{(2 k+2)!} \cdot \prod_{j=1}^k(1-4^j).
\]
\end{prop}

\begin{proof}
The map $F_{4}$ is conjugate to the Chebyshev polynomial $w \mapsto w^2-2$ via the change of variables $w=2-4 z$. The linearizing map $G_{4}$ satisfies
\[
G_{4}(z^2)=\frac{1-\cos (2 z)}{2}=\sum_{k \geq 1}(-1)^{k-1} \frac{(2 z)^{2 k}}{2(2 k)!}.\qedhere
\]
\end{proof}

\subsubsection{Values for  $\lambda=-2$}

\begin{prop}
For $k \in \N$,
\[
g_k(-2)= \begin{cases}\ds \frac{1}{(k+1)!} \left(-\frac{4}{3}\right)^{m} & \text { if } k =2m \text { is even } \\ \ds - \frac{1}{2(k+1)!} \left(-\frac{4}{3}\right)^{m+1} & \text { if } k=2m+1 \text { is odd }\end{cases}
\]
and
$$
a_k(-2)=g_k(-2) \cdot \prod_{j=1}^k(1-(-2)^j)
$$
\end{prop}

\begin{proof}
In that case, the map $F_{-2}$ is conjugate to the Chebyshev polynomial $w \mapsto w^2-2$ via the change of variables $w=2 z-1$. The linearizing map $G_{-2}$ satisfies
\begin{eqnarray*}
G_{-2}(z) & = &\frac{1}{2}+\cos \left(\frac{-2 \sqrt{3}}{3} z+\frac{2 \pi}{3}\right) \\
& = &\frac{1}{2}+\frac{\sqrt{3}}{2} \sin \left(\frac{2 \sqrt{3}}{3} z\right)-\frac{1}{2} \cos \left(\frac{2 \sqrt{3}}{3} z\right) \\
& =&\sum_{k \geq 0}\left(-\frac{4}{3}\right)^{k} \frac{z^{2 k+1}}{(2 k+1)!}- \sum_{k \geq 1}\left(-\frac{4}{3}\right)^{k} \frac{z^{2 k}}{2(2 k)!}.\qedhere
\end{eqnarray*}
\end{proof}

\end{document}